\documentclass[11pt,a4paper]{article}
\usepackage{graphicx}
\usepackage{amsfonts,amssymb,amsmath,amsthm}

\newtheorem{theorem}{Theorem}[section]
\newtheorem{lemma}[theorem]{Lemma}
\newtheorem{pro}[theorem]{Proposition}
\newtheorem{corol}[theorem]{Corollary}

\theoremstyle{definition}
\newtheorem{definition}[theorem]{Definition}

\theoremstyle{remark}
\newtheorem{remark}[theorem]{Remark}

\numberwithin{equation}{section}

\newcommand{\R}{\mathbb{R}}

\begin{document}
\title{Maximizing the spreading speed of KPP fronts in two-dimensional stratified media }
\author{Xing Liang \thanks{Department of Mathematics, University of Science and Technology
of China, xliang@ustc.edu.cn} \\
 \and Xiaotao Lin\thanks{Graduate School
 of Mathematical Sciences, University of Tokyo, linxt@ms.u-tokyo.ac.jp} \\
  \and Hiroshi Matano\thanks{Graduate School
 of Mathematical Sciences, University of Tokyo, matano@ms.u-tokyo.ac.jp} \\}

\date{}
\maketitle
\tableofcontents
\begin{abstract}
We consider the equation $u_t=u_{xx}+u_{yy}+b(x)f(u)+g(u)$,
$(x,y)\in\mathbb R^2$ with monostable nonliearity,
 where $b(x)$ is a nonnegative measure on $\mathbb R$ that is
 periodic in $x.$ In the case where $b(x)$ is a smooth periodic
 function, it is known that, for each $\theta\in[0,2\pi)$, there exists a \lq\lq{}planar'' travelling wave in the direction
 $\theta$ -- more precisely a pulsating travelling wave that propagates in the direction $(\cos\theta,\sin\theta)$ -- with
 average speed $c$ if and only if $c\geq c^*(\theta,b),$ where $c^*(\theta,b)$ is a certain
 positive number depending on $b.$ This constant $c^*(\theta,b)$ is called the \lq\lq minimal
 speed". Moreover, the quantity
 $w(\theta;{b})=\min_{|\theta-\phi|<\frac{\pi}{2}}c^*(\phi;{b})/\cos(\theta-\phi)$
 is called the spreading speed in direction $\theta$ in the
 following sense:
any solution $u(x,y,t)$ with nonnegative compactly supported initial
data $u_0(x,y)\not\equiv 0$ satisfies
\[
\ \lim_{t\to\infty}u(x,y,t)=0,\hbox{   on }\{x\cos\theta+y\sin\theta
>ct\}\,\,\,\hbox{  if }c>w(\theta;b),
\]
\[
\ \lim_{t\to\infty}u(x,y,t)=1,\hbox{   on
 }\{x\cos\theta+y\sin\theta <ct\}\,\,\,\hbox{  if }c<w(\theta;b).
\] This theory can be extended by showing the
 existence of the minimal speed $c^*(\theta,b)$ for any nonnegative
 measure $b$ with period $L.$ We then study the question of maximizing $c^*(\theta,b)$
 under the constraint $\int_{[0,L)}b(x)dx=\alpha L,$ where $\alpha$ is an arbitrarily given positive constant.
 We prove that
 the maximum is attained by periodically arrayed Dirac's delta
 functions $h(x)=\alpha L\sum_{k\in\mathbb Z}\delta(x+kL)$ for any direction $\theta$. Based on these results, for the case that $b=h$ we
also show the monotonicity of the spreading speedsin $\theta$ and
study the asymptotic shape of spreading fronts for large $L$ and
small $L$ . Finally, we show that for general 2-dimensional periodic
equation $u_t=u_{xx}+u_{yy}+b(x,y)f(u)+g(u)$, $(x,y)\in\mathbb R^2$,
the similar conclusions do not hold.
\end{abstract}

{\section{Introduction}}\label{Introduction} Travelling waves
describe a wide class of phenomena in combustion physics, chemical
kinetics, biology and other natural sciences. {}From the physical
point of view, travelling waves usually describe transition
processes. Transition from one equilibrium to another is a typical
case, although more complicated situations may arise. Since the
classical paper by Kolmogorov, Petrovsky and Piskunov in 1937,
travelling waves have been intensively studied. For example, the
monograph \cite{s2} provides a comprehensive discussion on this
subject.

{}From the ecological point of view, travelling waves typically
describe the expansion of the territory of a certain species,
including, in particular, the invasion of alien species in a given
habitat. Models for biological invasions in spatially periodic
environments were first introduced by Shigesada {\it{et al}}. in
dimensions 1 and 2 (see \cite{s4,s5,s6}). More precisely, they
considered spatially segmented habitats where favorable and less
favorable (or even unfavorable) zones appear alternately and
analyzed how the pattern and scale of spatial fragmentation affect
the speed of invasions. In their study, the spatial fragmentation
was typically represented by step functions which take two different
values periodically. Mathematically, their analysis was partly
unrigorous as it relied on formal asymptotics of the travelling wave
far away from the front. The works on more general spatially
periodic travelling waves in high dimensional case also can be found
in the survey paper of Xin \cite{xin} and the reference therein.

Berestycki, Hamel \cite{BH} and Berestycki, Hamel, Roques \cite{s1}
extended and mathematically deepened the work of Shigesada \emph{et
al}. significantly, by dealing with much more general equations of
the form $u_t=\nabla \cdot ((A(x)\nabla u))+f(x,u)$ in $\mathbb R^n$
with rather general smooth periodic coefficients and by developing
various mathematical techniques to study the effect of environmental
fragmentation rigourously.

Among other things, they proved that, under certain assumptions on
the coefficients, there exists $c^*>0$ such that the equation has a
pulsating travelling wave if and only if $c\geq c^*.$ Furthermore,
they showed that the minimal speed $c^*$ is characterized by the
following formula:
\[
\ c^*=\min\{c>0\,|\, \exists \lambda>0 \hbox{ such that
}\mu(c,\lambda)=0 \},
\]
where $\mu(c,\lambda)$ is the principal eigenvalue of a certain
elliptic operator associated with the linearization  of the
travelling wave far away from the front.

By using a totally different approach Weinberger \cite{s8} also
proved the existence of the minimal speed $c^*$ of pulsating
travelling waves in a more abstract framework. His method relies on
the theory of monotone operators and is a generalization of his
earlier work \cite{Wein1982} to spatially periodic media.

In Weinberger \cite{s8} and H. Berestycki, F. Hamel, N, Nadirashvili
\cite{BHN}, another important concept-spreading speed- is concerned.
Here the \lq\lq spreading speed'' roughly means the asymptotic speed
of an expanding front that starts from a compactly supported initial
data.

In \cite{llm}, we considered 1-dimensional problem
\[ \
u_t=u_{xx}+\bar{b}(x)u(1-u)
\]
and showed that the minimal speed $c^*(\bar{b})$ exists in the sense
that the travelling wave solutions exist for any nonnegative
$L$-periodic measure $\bar{b}$ and any speed $c\geq c^*(\bar{b})$,
where $L$ is a positive constant. $c^*(\bar{b})$ is also the
spreading speed. Moreover, the maximum of $c^*(\bar{b})$ is attained
by
\begin{equation}\label{h} \ h(x):=\alpha L\sum_{k \in\mathbb Z}
\delta\big(x-(k+\frac{1}{2})L\big),
\end{equation} where $\delta(x)$ is the Dirac's delta function.

In this paper, regarding to the original model of \cite{s6} we
consider the following KPP equation in two-dimensional stratified
media:
\begin{equation}\label{FirstEquation1}
u_t=u_{xx}+u_{yy}+{\bar{b}}(x)f(u)+g(u) \quad\ \ (x,y\in\mathbb R,)
\end{equation}and the Cauchy problem
\begin{equation}\label{FirstEquation}
\left\{\begin{array}{ll} u_t=u_{xx}+u_{yy}+{\bar{b}}(x)f(u)+g(u)
\quad\ \ &(x,y\in\mathbb R,\;t>0),\vspace{5pt}\\
u(x,y,0)=u_0(x,y)\quad\ \ &(x,y\in\mathbb R),
\end{array}\right.
\end{equation}
where $\bar{b}$ is either a smooth function or a measure satisfying
$\bar{b}(x)\geq 0$ and $\bar{b}(x+L)\equiv \bar{b}(x),\,x\in\mathbb
R$, for some constant $L>0$.

 Moreover, we suppose that
$f$ and $g$ satisfy one of the following two cases:
\begin{enumerate}
\item [{\rm(F1)}.] $g=0$, $f\in C^1(\mathbb R^+)$, $f(0)=f(1)=0$, $f(u)>0\,\,(0<u<1)$,
$f(u)<0\,\, (u>1)$, $f'(0)>0$, $f(u)/u$ is decreasing in $u>0$.
\item [{\rm(F2)}.] $f(u)=u$, $g(u)=-ug_1(u)$ and $g_1\in C^1(\mathbb R^+)$, $g_1(0)=0$, $g_1'(u)>0\,\,(u>0)$, $g_1(u)\to\infty$
(as $u\to\infty$).
\end{enumerate}
Therefore the equation \eqref{FirstEquation1} can be written as
\[
\ u_t=u_{xx}+u_{yy}+\bar{b}(x)f(u)
\] for case {\rm(F1)}, and
\[
\ u_t=u_{xx}+u_{yy}+u(\bar{b}(x)-g_1(u))
\] for case {\rm(F2)}. Typical examples include the following:
for {\rm(F1)}, $u_t=u_{xx}+u_{yy}+\bar{b}(x)u(1-u)$; for {\rm(F2)},
$u_t=u_{xx}+u_{yy}+u(\bar{b}(x)-u)$.

Applying the results of Berestycki, Hamel \cite{BH}, Berestycki,
Hamel, Roques \cite{s1} and Weinberger \cite{s8} it is easy to show
that when $\bar{b}$ is a smooth function, the minimal speed of the
travelling wave solution of \eqref{FirstEquation1} in direction
$\theta$ would be
\[
\ c^*(\theta, \bar{b})=\min \{c>0\,|\,\exists \lambda>0\hbox{ such
that } \mu(\lambda,\theta,\bar{b})=\lambda^2-\lambda c\},
\]
where $\mu(\lambda,\theta,\bar{b})$ is a quantity such that
$\exists\, \psi(x)>0$ satisfying
\[
\,\,\,\,-\psi_{xx}+2\lambda
\cos\theta\psi_x-\bar{b}(x)f'(0)\psi=\mu(\lambda,\theta,\bar{b})\psi
\] with $ \psi(x)\equiv\psi(x+L)$. Moreover, in direction
$\theta$ the spreading speed $w(\theta,\bar b)$ exists in the
following sense: any solution $u(x,y,t)$ with nonnegative compactly
supported initial data $u_0(x,y)\not\equiv 0$ satisfies
\[
\ \lim_{t\to\infty}u(x,y,t)=0,\hbox{   on }\{x\cos\theta+y\sin\theta
>ct\}\,\,\,\hbox{  if }c>w(\theta;b),
\]
\[
\ \lim_{t\to\infty}u(x,y,t)=P(x),\hbox{   on
 }\{x\cos\theta+y\sin\theta <ct\}\,\,\,\hbox{  if }c<w(\theta;b),
\] where $P$ is the unique positive steady state of
\eqref{FirstEquation1} which is $L$-periodic in $x$. Moreover,
$w(\theta,\bar
b)=\min_{|\theta-\phi|<\frac{\pi}{2}}c^*(\phi;\bar{b})/\cos(\theta-\phi).$

In this paper, we also consider the problem of maximizing the
minimal speed and the spreading speed. We will show in Theorem
\ref{cstarbequalcestarb2D} that the minimal speed of travelling wave
for equation \eqref{FirstEquation1} exists if $\bar{b}(x)$ is a
measure. Moreover, the spreading speed exists and can be expressed
as in \eqref{eq:sprspeed} (Theorem \ref{Themforspreadingspeed2D}).
We denote the minimal speed in direction $\theta$ by
$c^*(\theta;\bar{b})$ and denote the spreading speed in direction
$\theta$ by $w(\theta;\bar{b})$. We also show the monotonicity of
the spreading speed in the direction $\theta$ in Theorem
\ref{Th:monotone-theta}. Then we will show in Theorem
\ref{MinimalPlanarSpeed} that the maximum of
$\{c^*(\theta;\bar{b})\}$ and $\{w(\theta;\bar{b})\}$ is attained by
$\bar{b}=h$. Based on these results, for the case that $b=h$ we also
study the asymptotic shape of spreading fronts for large $L$ and
small $L$. We will show in Theorem \ref{Thm:Lto0}
 that, for small $L$, the asymptotic shape of the front
is a circle and show in Theorem \ref{thm:speedforlargel} that, for
large $L$, the asymptotic shape of the front is a parabola. At the
end of this paper, we consider a more general case that $b$ is
periodically dependent on $x$ and $y$. Then a sequence of $b_n$
exists such that the corresponding minimal speed $c^*(\theta;b_n)$
is convergent to $+\infty$ as $n\to\infty$.

This paper is organized as follows. In Section
\ref{Notationandmainresults}, we introduce the basic notations and
state the main results. In Section \ref{Eigenvalueproblem}, we
state properties of linear eigenvalue problem which will be used
in Section \ref{Proofofresults}. In Section
\ref{1Dimensionalcase}, we recall our results in \cite{llm} on the
1 dimensional problem and extend the results to more general
equations of the form
\begin{equation}\label{1-dim}
u_t=u_{xx}+\bar{b}(x)f(u)+g(u).
\end{equation} for both of the cases (F1) and (F2). In Section \ref{section:wellposedness}, we prove the
well-posedness of equation \eqref{FirstEquation} in uniform
topology. In Section \ref{Localuniformtopolyandsemiflow}, we
consider the the well-posedness of equation \eqref{FirstEquation} in
local uniform topology  and the theory of semiflow generated by the
solution of \eqref{FirstEquation} for the cases of (F1) and (F2).
These are used in Section \ref{Proofofresults}. In Section
\ref{Proofofresults}, we prove the main results. In Section
\ref{General2Dimensionalcase}, we consider the general 2-dimensional
case.

\section{Notation and main results}\label{Notationandmainresults}
\subsection{Basic notation}\label{subsection:notation}
In what follows we suppose that constants $L>0$ and $\alpha>0.$ Let
$\Lambda(\alpha)$ be the set defined by
\[
\ \Lambda(\alpha):=\{b \in C^1(\mathbb R)\,|\, b(x)\geq
0,b(x)=b(x+L) \,~and~ \int_{[0,L)} b(x)dx=\alpha L \}.
\]
\begin{definition}\label{DefofLamdaBarAlpha}
$\overline{\Lambda}(\alpha)$ is defined to be the sequential closure
of $\Lambda(\alpha)$ in the space of distribution on $\mathbb{R}.$
More precisely, $\bar{b}\in\overline{\Lambda}(\alpha)$ if and only
if there exists a sequence $\{b_n\}_{n=1}^{\infty}$ in
$\Lambda(\alpha)$ such that
\begin{equation}\label{DefinitionOfLamdabarAlpha}
\
\int_{\mathbb{R}}\bar{b}(x)\eta(x)dx=\lim_{n\to\infty}\int_{\mathbb{R}}b_n(x)\eta(x)dx
\end{equation}
for any test function $\eta\in C_0^\infty(\mathbb{R}),$ where the
left-hand side of \eqref{DefinitionOfLamdabarAlpha} is a formal
integration representing the dual product $\langle \bar{b},\eta
\rangle.$ 
\end{definition}
Since each $b_n$ is nonnegative, $\bar{b}$ is a nonnegative
distribution. Consequently, $\bar{b}$ is a Radon measure on
$\mathbb{R}.$ Therefore \eqref{DefinitionOfLamdabarAlpha} holds
for every $\eta\in C_0(\mathbb{R}).$ We say that $b_n\to\bar{b}$
in the weak$^*$ sense if \eqref{DefinitionOfLamdabarAlpha} holds.

In what follows, we will not distinguish the measure $\bar{b}$ and
its density function $\bar{b}(x),$ as long as there is no fear of
confusion. Thus we will often use expression as in the left-hand
side of \eqref{DefinitionOfLamdabarAlpha}.

We also note that, since $b_n(x)$ is $L$-periodic, $\bar{b}(x)$ is
also $L$-periodic in the following sense:
\begin{equation}\label{L-periodicity}
\ \int_\mathbb{R}\bar{b}(x)
\eta(x+L)dx=\int_{\mathbb{R}}\bar{b}(x)\eta(x)dx\,\,\, \hbox{ for
}\eta\in C_0(\mathbb{R}).
\end{equation}

The next Lemma \ref{lemma:A} and Lemma \ref{cor:A} are the same as
 Lemma 2.2 and Corollary 2.3 in \cite{llm}. We state them here
without proof.
\begin{lemma}\label{lemma:A}
Let $\{b_n\}\subset\Lambda(\alpha)$ be a sequence converging to some
$\bar{b}\in\overline{\Lambda}(\alpha)$ in the weak$^*$ sense. Let
$\eta(x)$ be a continuous function on $\mathbb{R}$ satisfying
\begin{equation}
\sum_{k=-\infty}^{\infty}\max_{0\leq x \leq L}|\eta(x+kL)|<\infty.
\end{equation}
Then $\eta$ is $\bar{b}$-integrable on $\mathbb{R}$ and the
following conclusion holds:
\begin{equation}\label{bn-to-b-R}
\lim_{n\to\infty}\int_{\mathbb{R}}b_n(x)\eta(x)dx=\int_{\mathbb{R}}\bar{b}(x)
\eta(x)dx,
\end{equation}
\begin{equation}\label{sum-eta}
\int_{\mathbb{R}}\bar{b}(x)|\eta(x)|dx \leq \alpha L
\sum_{k=-\infty}^{\infty}\max_{0\leq x \leq L}|\eta(x+kL)|.
\end{equation}
\end{lemma}
\begin{lemma}\label{cor:A} Let $\{b_n\}$, $\bar{b}$ and $\eta$ be as
in Lemma \ref{lemma:A} and let $\{f_n\}$ be a sequence of uniformly
bounded continuous functions on $\R$ converging to $f$ locally
uniformly on $\R$. Then
\[
\ \int_\R\bar{b}(x)f(x)\eta(x)dx=\lim_{n\to\infty}\int_\R
b_n(x)f_n(x)\eta(x)dx.
\]
\end{lemma}
Now we consider the equation\begin{equation}\label{SecondEquation2}
 u_t=u_{xx}+u_{yy}+\bar{b}(x)f(u)+g(u)
\quad\ \ (x,y\in\mathbb R,\;t>0)
\end{equation} and the corresponding Cauchy problem
\begin{equation}\label{SecondEquationcauchy}
\left\{\begin{array}{ll} u_t=u_{xx}+u_{yy}+\bar{b}(x)f(u)+g(u)
\quad\ \ &(x,y\in\mathbb R,\;t>0),\vspace{5pt}\\
u(x,y,0)=u_0(x,y)\quad\ \ &(x,y\in\mathbb R).
\end{array}\right.
\end{equation} Here $f(u)$ and $g(u)$ are locally Lipschitz
continuous functions.
\begin{definition}\label{def:wsolg}
Let $I\subset \mathbb R$ be any open interval. A continuous function
$u(x,y,t):\mathbb R^2\times I\to \mathbb R^2$ is called a
{\textbf{weak solution}} of \eqref{SecondEquationcauchy} for $t\in
I$ (or a {solution} in the weak sense) if for any $\eta(x,y,t)\in
C_0^\infty(\mathbb R^2\times I),$
\[
\ -\int_I\int_{\mathbb R}\int_{\mathbb R}u
\eta_t\,dxdydt=\int_I\int_{\mathbb
R^2}\big(u\eta_{xx}+u\eta_{yy}+\bar{b}(x)f(u)\eta+g(u)\eta\big)\,
dxdydt.
\]
\end{definition}
\begin{definition}\label{def:msolg}
A continuous function $u(x,y,t):\mathbb R^2\times(0,T)\to \mathbb R$
for some constant $T\in (0,\infty]$ is called a {\textbf{mild
solution}} of \eqref{SecondEquationcauchy} if
\[
\ \lim_{t\searrow 0}u(x,y,t)=u_0(x,y)\,\hbox{ for any }x,\,y \in
\mathbb R
\]
and if it can be written as
\[
\begin{split}
 u(x,y,t)=&\int_{\mathbb
R^2}G(x-x',y-y',t)u_0(x',y')dx'dy'  \\
&+\int_0^t\int_{\mathbb
R^2}G(x-x',y-y',t-s)\big(\bar{b}(x')f(u(x',y',s))+g(u(x',y',s))\big)\,dx'dy'ds,
\end{split}
\]
where
\begin{equation}\label{heatkernel}
G(x,y,t):=\frac{1}{{4\pi t}}\exp\big(-\frac{x^2+y^2}{4t}\big).
\end{equation}
\end{definition}
As we will show in Section \ref{section:wellposedness}, a local mild
solution of \eqref{SecondEquationcauchy} exists uniquely for any
$\bar{b}\in\overline{\Lambda}(\alpha)$ and $u_0\in C(\mathbb
R^2)\cap L^\infty(\mathbb R^2)$ with $u_0\geq 0$. Furthermore, any
mild solution is a weak solution.

It is easily seen that, in the case of $T=\infty$, if $u(x,y,t)$ is
a mild solution of \eqref{SecondEquationcauchy}, then for any
constant $\tau\geq 0,$ $u(x,y,t+\tau)$ is a mild solution of
\eqref{SecondEquationcauchy} with initial data $u(x,y,\tau)$.

We call a function $u(x,y,t)$ on $\mathbb{R}^3$ a \textbf{mild
(entire) solution} of \eqref{SecondEquation2} if, for any
$\tau\in\mathbb{R},$ $u(x,y,t+\tau)$ is mild solution of
\eqref{SecondEquationcauchy} with initial data
$u_0(x,y)=u(x,y,\tau).$

 In this paper,
we study the minimal speed of travelling waves and the spreading
speed for monostable type nonlinearities. More precisely, we
consider the following two cases:
\begin{enumerate}
\item [{\rm(F1)}.] $g=0$, $f\in C^1(\mathbb R^+)$, $f(0)=f(1)=0$, $f(u)>0\,\,(0<u<1)$,
$f(u)<0\,\, (u>1)$, $f'(0)>0$, $f(u)/u$ is decreasing in $u>0$.
\item [{\rm(F2)}.] $f(u)=u$, $g(u)=-ug_1(u)$ and $g_1\in C^1(\mathbb R^+)$, $g_1(0)=0$, $g_1'(u)>0\,\,(u>0)$, $g_1(u)\to\infty$
(as $u\to\infty$).
\end{enumerate}
The equation \eqref{SecondEquation2} can then be written as
\[
\ u_t=u_{xx}+u_{yy}+\bar{b}(x)f(u)
\] for case {\rm(F1)}, and
\[
\ u_t=u_{xx}+u_{yy}+u(\bar{b}(x)-g_1(u))
\] for case {\rm(F2)}. Typical examples include the following:
for {\rm(F1)}, $u_t=u_{xx}+u_{yy}+\bar{b}(x)u(1-u)$; for {\rm(F2)},
$u_t=u_{xx}+u_{yy}+u(\bar{b}(x)-u)$.

For both of the cases {\rm(F1)} and {\rm(F2)} with $\bar{b}$
replaced by a smooth $b$, it is known that (see \cite{s8} etc.)
any travelling wave $u$ in the direction $\theta$ has the
following asymptotic expression in the \lq\lq leading edge",
namely the area where $u\approx 0$:
\begin{equation}\label{AsymptoticExpression}
u(x,y,t)\sim e^{-\lambda(x\cos\theta+y\sin\theta-ct)}\psi(x),
\end{equation} where $\psi(x+L)\equiv \psi(x)>0,$ and $\lambda>0$
is some constant. Substituting \eqref{AsymptoticExpression} into
equation (\ref{SecondEquationcauchy}), we obtain the identity
\begin{equation}\label{IdentityForB}
-\psi''(x)+2\lambda\cos\theta\psi'(x)-b(x)f'(0)\psi(x)=(\lambda^2-\lambda
c )\psi(x).
\end{equation}
\begin{definition}\label{definition for L} For $\bar{b}\in\overline{\Lambda}(\alpha),$ $\lambda\in\R,$ we define
an (unbounded) operator $-L_{\lambda,\theta,\bar{b}}$ on the Banach
space $\{\psi\in C(\mathbb{R})\,|\,\psi(x)=\psi(x+L)\}$ with
$\|\psi\|=\max_{x\in \mathbb{R}}|\psi(x)|$ as follows:
\begin{equation}\label{operator-L}
-L_{\lambda,\theta,\bar{b}}\psi(x)=-\psi''(x)+2\lambda\cos\theta\psi'(x)-\bar{b}(x)f'(0)\psi(x).
\end{equation} Here the derivatives are understood in the \lq\lq weak sense" by
which we mean that $-L_{\lambda,\theta,\bar{b}}\psi=g$ if and only
if, for any $\varphi\in C_0^{\infty}(\mathbb{R}),$
\[
\ \int_\mathbb{R}(-\varphi''-2\lambda\cos
\theta\varphi'-\bar{b}f'(0)\varphi)\,\psi\,
dx=\int_{\mathbb{R}}\varphi g \,dx.
\]
\end{definition}
%
\begin{definition}\label{definitionofprincipaleigenvalue}
Given $\bar{b}\in\overline{\Lambda}(\alpha)$, $\lambda\in\R$, and
$\theta\in [0,2\pi)$, we call $\mu(\lambda,\theta,\bar{b})$ the
{\textbf{principal eigenvalue}} of the operator
$-L_{\lambda,\theta,\bar{b}}$ in \eqref{operator-L} if there exists
a positive continuous function $\psi$ with $\psi(x)\equiv\psi(x+L)$
satisfying
\begin{equation}\label{identityofmulambdabbar}
-\psi''(x)+2\lambda\cos\theta\psi'(x)-\bar{b}(x)f'(0)\psi(x)=\mu(\lambda,\theta,\bar{b})\psi
\end{equation} in the weak sense. Here $\psi$ is called the
\textbf{principal eigenfunction}.
\end{definition}
It is clear that the following identity holds for any $\lambda\in\R$
and $\theta\in[0,2\pi)$:
\begin{equation}\label{eigenvalueidentity}
\mu(\lambda,\theta,\bar{b})=\mu_0(\lambda\cos\theta,\bar{b}),
\end{equation}
where $\mu_0(\lambda,\bar{b})$ denotes the principal eigenvalue of
the operator
\[
-L_{\lambda,0,\bar{b}}\psi=-\psi''+2\lambda\psi'-\bar{b}(x)f'(0)\psi.
\] Various properties of $\mu_0(\lambda,\bar{b})$ have been
established in \cite{llm} to study the 1 dimensional problem
\eqref{1-dim}. Combining \eqref{eigenvalueidentity} and Proposition
2.13 in \cite{llm}, one immediately obtains the following:
\begin{pro}\label{principaleigenvalueisunique}
For any $\bar{b}\in\overline{\Lambda}(\alpha)$, $\theta\in[0,2\pi)$
and $\lambda\in\R,$ the principal eigenvalue
$\mu(\lambda,\theta,\bar{b})$ exists. It is unique and simple, and
belongs to $\mathbb{R}$.
\end{pro}
\begin{remark}\label{rem:f'(0)=1}
For notational simplicity, we will assume $f'(0)=1$ in Section
\ref{Eigenvalueproblem} and later sections. We can do it without
loss of generality since otherwise we can simply replace $f(u)$ by
$f(u)/f'(0)$, $\bar{b}(x)$ by $f'(0)\bar{b}(x)$ and $\alpha$ by
$f'(0)\alpha$.
\end{remark}
In order to study the travelling wave solutions, a positive periodic
bounded steady-state of \eqref{SecondEquationcauchy} is
indispensable. We call a $t$-independent mild solution $P(x,y)$ the
{\bf steady-state} of \eqref{SecondEquationcauchy} or we say
$P(x,y)$ satisfies $u_{xx}+u_{yy}+\bar{b}(x)f(u)+g(u)=0$ in the mild
sense. Moreover, we say a function $\phi(x,y)$ is $L$-periodic
provided $\phi(x,y)\equiv \phi(x+L,0)$. In this paper, we will show
that for any $\bar{b}\in\overline{\Lambda}(\alpha)$, for both of
cases {\rm(F1)} and {\rm(F2)}, the $L$-periodic positive
steady-state mentioned above exists. For case {\rm(F1)}, $P\equiv
1$. For case {\rm(F2)}, see Lemma \ref{thm:monostability}.
\begin{definition}\label{def:speedforbx}
A mild entire solution of \eqref{SecondEquation2} is called a
travelling wave in the direction $\theta$ with average speed $c$ if
it can be written as
\[
\ u(x,y,t)=\varphi(x,x\cos\theta+y\sin\theta-ct)
\] where $\varphi$ satisfies $\varphi(x+L,s)\equiv\varphi(x,s)$
and
\[
\ \lim_{s\to\infty}\varphi(x,s)=0,\lim_{s\to
-\infty}\varphi(x,s)=P(x),
\]where $P$ is the $L$-periodic positive
steady-state.
\end{definition}
\begin{definition} Given $\bar{b}\in\overline{\Lambda}(\alpha)$,
we define the minimal speed $c^*(\theta;\bar{b})$ and a quantity
$c_e^*(\theta;\bar{b})$ which is related to the minimal speed by
\[\left\{
    \begin{array}{ll}
      c^*(\theta;\bar{b}):=\inf\{c>0\,|\,\hbox{Travelling wave with speed $c$
in direction $\theta$ exists}\}, \\
      c_e^*(\theta;\bar{b}):=\inf\{c>0\,|\,\exists\,\lambda>0,\,\hbox{such
that }\mu(\lambda,\theta,{\bar{b}})=\lambda^2-\lambda c\}.
    \end{array}
  \right.
\
\]
\end{definition}
\begin{definition}
Given $\bar{b}\in\overline{\Lambda}(\alpha)$, a quantity
$w(\theta;\bar{b})$ is called the spreading speed in direction
$\theta$ if solution $u(x,y,t)$ of \eqref{SecondEquationcauchy} with
nonnegative compactly supported initial data $u_0(x,y)\not\equiv 0$
satisfies
\[
\ \lim_{t\to\infty}u(x,y,t)=0,\hbox{   on }\{x\cos\theta+y\sin\theta
>ct\}\,\,\,\hbox{  if }c>w(\theta;\bar{b}),
\]
\[
\ \lim_{t\to\infty}u(x,y,t)=P(x),\hbox{   on
 }\{x\cos\theta+y\sin\theta <ct\}\,\,\,\hbox{  if }c<w(\theta;b)
\]
\end{definition} where $P$ is the $L$-periodic positive
steady-state. We will use Weinberger \cite{s8} to show that
$c^*(\theta;\bar{b})=c^*_e(\theta;\bar{b})$ and travelling wave with
speed $c$ exists for $c>c^*(\theta;\bar{b})$. Moreover
$w(\theta;\bar{b})$ exists for
$\bar{b}\in\overline{\Lambda}(\alpha)$ and
$w(\theta;\bar{b})=\min_{|\theta-\phi|<\frac{\pi}{2}}c^*(\phi;\bar{b})/\cos(\theta-\phi)$.

\subsection{Main results}\label{subsection:MainResults}

In the following theorems
\ref{cstarbequalcestarb2D}--\ref{thm:speedforlargel}, we assume that
either {\rm(F1)} or {\rm(F2)} holds:

\begin{theorem}[Minimal speed]\label{cstarbequalcestarb2D}
For any $\bar{b}\in\overline{\Lambda}(\alpha)$, $\theta\in[0,2\pi)$,
it holds that $c^*(\theta;\bar{b})>0,$ and that a travelling wave in
the direction $\theta$ with speed $c$ exists if and only if $c\geq
c^*(\theta,\bar{b})$. Furthermore,
\[
c^*(\theta;\bar{b})=c_e^*(\theta;\bar{b}).
\]
\end{theorem}

\begin{corol}\label{cr:directional-symmetry}
For any $\bar{b}\in\overline{\Lambda}(\alpha)$,
\[
c^*(-\theta;\bar{b})=c^*(\theta;\bar{b}), \quad\ \
c^*(\theta+\pi;\bar{b})=c^*(\theta;\bar{b}),
\]
where the value of $\theta$ is understood in the ${\rm mod}~2\pi$
sense.
\end{corol}

\begin{theorem}[Continuity]\label{Convergenceofcstarbn}
For any $\bar{b}\in\overline{\Lambda}(\alpha)$, $\theta\in[0,2\pi)$,
if $\{b_n\}$ is a sequence in $ \Lambda(\alpha)$ converging to
$\bar{b}$ in the weak$^*$ sense, then
\[
c^*(\theta;\bar{b})=\lim_{n\to\infty}c^*(\theta;b_n).
\]
Furthermore, 
\begin{equation}\label{boundednessinequality}
2\sqrt{\tilde{\alpha}}\leq c^*(\theta,\bar{b})\leq
2\sqrt{\tilde{\alpha}+\tilde{\alpha}^2L^2},
\end{equation} where $\tilde{\alpha}=f'(0)\alpha$.
\end{theorem}

\begin{theorem}[Spreading speed]\label{Themforspreadingspeed2D}
For any $\bar{b}\in\overline{\Lambda}(\alpha)$, $\theta\in[0,2\pi)$,
the spreading speed in the direction $\theta$, denoted by
$w(\theta;\bar{b})$, exists and
\begin{equation}\label{eq:sprspeed}
\ w(\theta;\bar{b})=\min_{|\theta-\phi|<\frac{\pi}{2}}
c^*(\phi;\bar{b})/\cos(\theta-\phi).
\end{equation}
\end{theorem}

\begin{theorem}[Monotonicity in $\theta$]\label{Th:monotone-theta}
For any nonconstant $\bar{b}\in\overline{\Lambda}(\alpha)$, both
$c^*(\theta,\bar{b})$ and $w(\theta;\bar{b})$ are strictly monotone
increasing in $\theta\in[0,\pi/2]$; hence srtictly monotone
decreasing in $\theta\in[-\pi/2,0]$.
\end{theorem}

The next result deals with the problem of maximizing the spreading
speed among all $\bar{b}\in\overline{\Lambda}(\alpha)$ for a given
$\alpha>0$. Here $h$ denotes the measure on $\R^2$ given by
\eqref{h}. It represents mass concentration on the parallel lines
$x=(k+\frac12)L,\;k\in{\mathbb Z}.$

\begin{theorem}[Optimal coefficient]\label{MinimalPlanarSpeed}
For any $\theta\in [0,2\pi)$,
\[
c^*(\theta;h)>c^*(\theta;b)\quad \hbox{for any}\ \;
{b\in\Lambda(\alpha)},
\]
\[
c^*(\theta;h)=\max_{\bar{b}\in\overline{\Lambda}(\alpha)}
c^*(\theta;\bar{b})=\sup_{b\in\Lambda(\alpha)}c^*(\theta;b).
\]
Consequently, the spreading speed of $h$ satisfies
\[
w(\theta;h)=\max_{\bar{b}\in\overline{\Lambda}(\alpha)}
w(\theta;\bar{b})=\sup_{b\in\Lambda(\alpha)}w(\theta;b)\ \ \hbox{for
any}\ \;\theta\in[0,2\pi).
\]
\end{theorem}

The following theorems are concerned with the asymptotics of the
speeds for large and small $L$.  In order to emphasize the
dependence of the optimal speeds $c^*(\theta;h)$ and $w^*(\theta;h)$
on $L$, we write:
\[
c^*(\theta;\alpha,L):=c^*(\theta;h_{\alpha,L}),\quad\
w(\theta;\alpha,L):=w(\theta;h_{\alpha,L}),
\]
where $h_{\alpha,L}$ denotes the measure defined in \eqref{h}.

\begin{theorem}\label{Thm:Lto0}
For any $\theta\in [0,2\pi)$ and $\alpha>0$,
\[
\ \lim_{L\to0}c^*(\theta;\alpha,L)=2\sqrt{\tilde{\alpha}}, \quad\ \
\lim_{L\to0}w(\theta;\alpha,L)=2\sqrt{\tilde{\alpha}},
\] where $\tilde{\alpha}=f'(0)\alpha$.
\end{theorem}

\begin{theorem}\label{thm:speedforlargel}
For any $\theta\in [0,2\pi)$ and $\alpha>0$,
\[
\ \lim_{L\to\infty}\frac{c^*(\theta;\alpha,L)}{L}=\left\{%
\begin{array}{ll}\hbox{\small$\displaystyle
  \frac{\tilde{\alpha}}{2\cos\theta}$}\,\,\,\, &
  \hbox{ if }\cos^2\theta\geq\frac{1}{2}; \vspace{6pt}\\
    \tilde{\alpha}\sin\theta\,\,\,\, &
  \hbox{ if }\cos^2\theta\leq\frac{1}{2}. \\
\end{array}%
\right.
\]
Therefore
\begin{equation}\label{parabolas}
\lim_{L\to\infty}\frac{w(\theta;\alpha,L)}{L}=\frac{\tilde{\alpha}}
{1+|\cos\theta|},\quad\hbox{where}\ \ \tilde{\alpha}=f'(0)\alpha.
\end{equation}
\end{theorem}

\begin{remark}\label{rm:Lto0}
Note that the limit value $2\sqrt{\tilde{\alpha}}$ in Theorem
\ref{Thm:Lto0} coincides with the speed for the homogeneous problem
where $b\equiv\alpha$.  Intuitively this is no surprise, since
$h_{\alpha,L}\to\alpha$ as $L\to 0$ in the weak$^*$ sense.
\end{remark}

\begin{remark}\label{rm:parabola}
Since the position of the spreading front at time $t$ is
approximated by $w(\theta;h)t$, the formula \eqref{parabolas}
implies that the asymptotic shape of the spreading front is roughly
expressed by a pair of parabolic caps when $L$ is very large .  Note that this parabolic shape appears
regardless of the choice of nonlinearity $f$. The speed in the
longitudinal direction is precisely twice as large as that in the
transverse direction.
\end{remark}

\section{Analysis of the eigenvalue problem}\label{Eigenvalueproblem}
As we have mentioned in Remark \ref{rem:f'(0)=1}, in what follows we
will always assume $f'(0)=1$ for notational simplicity, therefore
the eigenvalue problem \eqref{identityofmulambdabbar} can be written
as follows:
\[
\ -\psi''(x)+2\lambda\cos\theta\psi'(x)-
\bar{b}(x)\psi(x)=\mu(\lambda,\theta,\bar{b})\psi(x).
\]
The following proposition and lemmas are the immediate extensions of
Proposition 4.7 and Corollary 4.8 and Lemma 4.4 in \cite{llm}, and
they follow immediately from these results and
\eqref{eigenvalueidentity}.
\begin{pro}\label{pro:convergenceofeigenvalue}
Let $b_n$ be a sequence in $\Lambda(\alpha)$ converging to some
$\bar{b}$ in the weak$^*$ sense, and let $\lambda_n\to
\lambda\in\mathbb{R}, \theta_n\to \theta \in [0,2\pi].$ Then
\[
\mu(\lambda_n,\theta_n,b_n)\to\mu(\lambda,\theta,\bar{b})\,\,\,\,\,\,\hbox{
as }n\to\infty.
\]
\end{pro}
\begin{lemma}\label{lem:boundednessofmu}
For any $\lambda\in \mathbb{R},\theta\in[0,2\pi)$ and
$b\in\Lambda(\alpha)$, it holds that
\begin{equation}\label{estimateformu}
-\alpha\geq \mu(\lambda,\theta,b)\geq -\alpha-\alpha^2L^2.
\end{equation}
\end{lemma}
\begin{lemma}\label{LinearPsi}
There exists a constant $F>0$ such that for any $\lambda\in
\mathbb{R}$, $\theta\in[0,2\pi)$, $b\in\Lambda(\alpha),$ the
principal eigenfunction $\psi$
 of the operator $-L_{\lambda,\theta,b}$ satisfies
\[ \frac{\max\psi}{\min\psi}\leq F. \]
\end{lemma}
The next lemma is also easy to derive:
\begin{lemma}\label{lem:eigenvalueofhwiththeta}
Let $h$ be as in \eqref{h}. Then
\[
\left\{
  \begin{array}{ll}
    \mu(\lambda,\theta,h)<\mu(\lambda,\theta,b),\,\, & \forall\,b\in\Lambda(\alpha) \\
    \mu(\lambda,\theta,h)\leq\mu(\lambda,\theta,\bar{b}),\,\, & \forall\,
\bar{b}\in\overline{\Lambda}(\alpha).
  \end{array}
\right.
\]
\end{lemma}
\begin{proof}
One can find the inequality $\mu_0(\lambda,h)<\mu_0(\lambda,b)$ in
the proof of Lemma 4.12 in \cite{llm}. The first inequality of this
lemma follows immediately from the above inequality and the identity
\eqref{eigenvalueidentity}. Then second inequality then follows by
combining the first inequality with Proposition
\ref{pro:convergenceofeigenvalue}.
\end{proof}
\section{Minimal speed and spreading speed for the 1 dimensional
case}\label{1Dimensionalcase} In \cite{llm}, we treated the equation
$u_t=u_{xx}+\bar{b}(x)u(1-u)$, where
$\bar{b}(x)\in\overline{\Lambda}(\alpha)$. In this section, we
recall some results for this 1 dimensional case, which we have
proven in \cite{llm}, and we extend those results to the equation
\begin{equation}\label{eq:general1Dequation}
\ u_t=u_{xx}+\bar{b}(x)f(u)+g(u),\,\,x\in\mathbb R.
\end{equation} Without loss of generality, we may suppose that $f'(0)=1$
in case {\rm(F1)}, then the cases {\rm(F1)} and {\rm(F2)} have the
same linearization with the steady-state where $u\equiv 0$.

Note that \eqref{eq:general1Dequation} is a special case of
\eqref{SecondEquation2} since any solution of
\eqref{eq:general1Dequation} can be regarded as a $y$-independent
solution of \eqref{SecondEquation2}. Thus Definitions
\ref{def:wsol}, \ref{def:msol} and \ref{Definitionoftravellingwave}
below are special cases of Definitions \ref{def:wsolg},
\ref{def:msolg} and
\ref{def:speedforbx}. 

\begin{definition}\label{def:wsol}
Let $I\subset \mathbb R$ be any open interval. A continuous function
$u(x,t):\mathbb R\times I\to \mathbb R$ is called a {\textbf{weak
solution}} of \eqref{eq:general1Dequation} for $t\in I$ (or a
{solution} in the weak sense) if for any $\eta(x,t)\in
C_0^\infty(\mathbb R\times I),$
\[
\ -\int_I\int_{\mathbb R}u \eta_t\,dxdt=\int_I\int_{\mathbb
R}\big(u\eta_{xx}+\big(\bar{b}(x)f(u)+g(u)\big)\eta\big)\, dxdt.
\]
\end{definition}
We next consider the following Cauchy problem:
\begin{equation}\label{eq:general1Dequationcauchy}
\left\{\begin{array}{ll} u_t=u_{xx}+\bar{b}(x)f(u)+g(u)
\quad\ \ &(x\in\mathbb R,\;t>0),\vspace{5pt}\\
u(x,0)=u_0(x)\geq 0\quad\ \ &(x\in\mathbb R),
\end{array}\right.
\end{equation}
where $u_0\in C(\mathbb R)\cap L^\infty(\mathbb R).$
\begin{definition}\label{def:msol}
A continuous function $u(x,t):\mathbb R\times(0,T)\to \mathbb R$,
where $T\in(0,\infty]$ is a constant, is called a {\textbf{mild
solution}} of \eqref{eq:general1Dequationcauchy} if
\[
\ \lim_{t\searrow 0}u(x,t)=u_0(x)\,\hbox{ for any }x \in \mathbb R
\]
and if it can be written as
\[
\begin{split}
 u(x,t)=&\int_{\mathbb
R}G(x-y,t)u_0(y)dy  \\ &+\int_0^t\int_{\mathbb
R}G(x-y,t-s)\big(\bar{b}(y)f(u(y,s))+g(u(y,s))\big)\,dyds,
\end{split}
\]
where
\begin{equation}\label{heatkernel1D}
G(x,t):=\frac{1}{\sqrt{4\pi t}}\exp\big(-\frac{x^2}{4t}\big).
\end{equation}
\end{definition}
We have the following proposition for the well-posedness of
\eqref{SecondEquationcauchy}:
\begin{pro}\label{wlposedness} Suppose that $f,g$ are local
Lipschitz continuous. Then for any
$\bar{b}\in\overline{\Lambda}(\alpha)$ and any given nonnegative
initial data $u_0\in C(\mathbb R)\cap L^\infty(\mathbb R),$ there
exists a constant $T\in(0,+\infty]$, such that the problem
\eqref{eq:general1Dequationcauchy} has a unique mild solution
$u(x,t)$ for $t\in(0,T)$. This mild solution is also a weak
solution.
\end{pro}
Furthermore, we will prove the proposition of the existence of the
positive steady-state. Since we will prove it for 2-dimensional case
in Lemma \ref{thm:monostability}, we just state it here without
proof.
\begin{pro}\label{pro:steadystate}
For both of the cases {\rm(F1)} and {\rm(F2)} and any $\bar b\in\bar
\Lambda(\alpha)$ and $u_0\in C(\R)\cap L^\infty(\R)$, the mild
solution of \ref{eq:general1Dequationcauchy} with the initial data
$u_0$ exists for $t\in (0,\infty)$. Moreover, there exists a
positive $L-$periodic steady state $P(x)$ in the mild sense. The
solution $u(x,t,u_0)$ of the problem
\eqref{eq:general1Dequationcauchy} with initial data $u_0$ with
$u_0(x)=u_0(x+L)$,$u_0(x)\geq 0$ and $u_0\not\equiv 0$, satisfies
\[
\ \lim_{t\to\infty}u(x,t,u_0)=P(x)
\] uniformly for $x\in \R$.
\end{pro}
We call a function $u(x,t)$ on $\mathbb{R}\times\mathbb{R}$ a
\textbf{mild solution} of \eqref{eq:general1Dequation} for
$t\in\mathbb{R}$ if, for any $\tau\in\mathbb{R},$ $u(x,t+\tau)$ is
mild solution of \eqref{eq:general1Dequationcauchy} with initial
data $u_0(x)=u(x,\tau).$
\begin{definition}\label{Definitionoftravellingwave}
A mild solution $u(x,t)$ of \eqref{eq:general1Dequation} for
$t\in\mathbb R$ is called a {\textbf{travelling wave}} (in the
\textbf{positive} direction) if
\[
\
u(x,t)=\varphi(x,x-ct),\,\,\varphi(x+L,s)=\varphi(x,s)\,\,\,\,\,\,\,\,\,\,\,\,\hbox{
for } (x,t)\in\mathbb R\times\mathbb R,
\]
\[ \ \lim_{s  \to -\infty}\varphi(x,s)-P(x)=0,~~\lim_{s
\to +\infty}\varphi(x,s)=0 \quad \hbox{locally uniformly in }\,
x\in\mathbb R.
\]It is called a \textbf{travelling wave} (in the
\textbf{negative} direction) if
\[
\
u(x,t)=\varphi(x,x+ct),\,\,\varphi(x+L,s)=\varphi(x,s)\,\,\,\,\,\,\,\,\,\,\,\,\hbox{
for } (x,t)\in\mathbb R\times\mathbb R,
\]
\[ \ \lim_{s  \to -\infty}\varphi(x,s)=0,~~\lim_{s
\to +\infty}\varphi(x,s)-P(x)=0 \quad \hbox{locally uniformly in
}\,x\in\mathbb R.
\] The quantity $c$ is called the wave speed.
\end{definition}
The following theorem shows that  the minimal speed of travelling
wave in the positive direction and that in the negative direction
exist and are equal:
\begin{theorem}\label{cstarbequalcestarb}
For both of the cases {\rm(F1)} and {\rm(F2)} and any
$\bar{b}\in\overline{\Lambda}(\alpha),$ it holds that there is some
$c^*(\bar{b})>0$ such that a travelling wave of
\eqref{eq:general1Dequation} in the positive direction with speed
$c$ exists if and only if\, $c\geq c^*(\bar{b})$. The same holds for
travelling waves in the negative direction with speed $c$.
Furthermore,
\[
\
c^*(\bar{b})=\min_{\lambda>0}\frac{-\mu(\lambda,0,\bar{b})+\lambda^2}{\lambda},
\]
and consequently \[2\sqrt{\alpha}\leq c^*(\bar b)\leq
2\sqrt{\alpha+\alpha^2L^2}.\]
\end{theorem}
We call $c^*(\bar b)$ the {\bf minimal wave speed}. Moreover,
$c^*(\bar b)$ is also the {\bf spreading speed} in the following
sense:\begin{theorem}[Spreading speed]\label{Themforspreadingspeed}
 For both of the cases {\rm(F1)} and {\rm(F2)} and any $\bar{b}\in\overline{\Lambda}(\alpha),$
given a nonnegative initial data $u_0\not\equiv 0$ with compact
support, the mild solution $u(x,t,u_0)$ of
\eqref{eq:general1Dequation} satisfies that
\begin{enumerate}
\item[{\rm(i)}] $\lim\limits_{t\to \infty}u(x,t,u_0)=0  $\,\,
uniformly in $\{|x|>ct\}$ \,\,\,\,if \,\,$c> c^{*}(\bar{b})$,
\item[{\rm(ii)}] $\lim\limits_{t\to \infty}u(x,t,u_0)-P(x)=0  $\,\,
uniformly in $\{|x|<ct\}$ \,\,\,\,if\,\, $0<c< c^{*}(\bar{b})$.
\end{enumerate}
\end{theorem}
\begin{theorem}[Optimal coefficient]\label{cestarhlargerthancestarb}
Let $h$ be as in \eqref{h}. Then
\[
\
c^*(h)=\sup_{b\in\Lambda(\alpha)}c^*(b)=\max_{\bar{b}\in\overline{\Lambda}(\alpha)}c^*(\bar{b}).
\] Furthermore,
\[
\ c^*(h)>c^*(b) \quad\hbox{ for any }\,b\in\Lambda(\alpha).
\]
\end{theorem}

In fact, in \cite{llm}, we have proven the above theorems for
equation $u_t=u_{xx}+\bar{b}u(1-u)$. We can also use the same idea
to prove them for both of the cases {\rm(F1)} and {\rm(F2)}.
Moreover, we can see that any mild solution $u(x,t)$ of
\eqref{eq:general1Dequationcauchy} is also the mild solution of
\eqref{SecondEquationcauchy} independent of  $y$, and any travelling
wave of {1-dimensional equation} is also the travelling wave of
{2-dimensional equation} in the direction $\theta=0$ or $\pi$.
Hence, the above theorems for 1-dimensional case are covered by our
general results for 2-dimensional case. Hence, we omit the proof.
\section{Well-posedness for the 2-dimensional
case}\label{section:wellposedness} We recall the following Cauchy
problem \eqref{SecondEquationcauchy}:
\[
\left\{\begin{array}{ll} u_t=u_{xx}+u_{yy}+\bar{b}(x)f(u)+g(u)
\quad\ \ &((x,y)\in\mathbb R^2,\;t>0),\vspace{5pt}\\
u(x,y,0)=u_0(x,y)\geq 0\quad\ \ &((x,y)\in\mathbb R^2),
\end{array}\right.
\]
where $u_0\in C(\mathbb R^2)\cap L^\infty(\mathbb R^2)$ and $f$ and
$g$ are locally Lipschitz continuous functions of $u$.

We have the following theorems about the well-posedness of the mild
solution of \eqref{SecondEquationcauchy}:
\begin{theorem}\label{wlposedness} For any $\bar{b}\in\overline{\Lambda}(\alpha)$ and any given nonnegative
initial data $u_0\in C(\mathbb R^2)\cap L^\infty(\mathbb R^2),$
there exists a constant $T\in(0,+\infty]$, such that the problem
\eqref{SecondEquationcauchy} has a unique mild solution
$u(x,y,t,u_0,\bar{b})$ for $t\in(0,T)$. This mild solution is also a
weak solution. Furthermore, it is continuously dependent on the
initial data $u_0$. More precisely, if $u^0_n\to u_0$ in $L^\infty$
norm, then for any $t>0$, $u(x,y,t,u^0_n,\bar{b})\to
u(x,y,t,u_0,\bar{b})$ in $L^\infty$ norm.
\end{theorem} To prove the above theorem for general $f(u)$ and
$g(u)$, the following lemma is necessary:
\begin{lemma}\label{Equicontinuity}
For any $u_0\in C(\mathbb R^2)\cap L^\infty(\mathbb R^2)$, there is
some constant $T\in(0,\infty)$ such that the solution
$u(x,y,t,u_0,b)$  of \eqref{SecondEquationcauchy} with $\bar{b}$
replaced by a smooth $b\in\Lambda(\alpha)$ with initial data $u_0$
exists for $t\in(0,T)$. Moreover, for any $\epsilon>0, \,M>0,$
$\{u(x,y,t,u_0,b)\}_{u_0\in C(\mathbb R^2),\|u_0\|_\infty\leq M,
b\in\Lambda(\alpha)} $ is uniformly equicontinuous in $(x,y,t)\in
\mathbb{R}^2\times[\epsilon,T]$.
\end{lemma}
For the special cases {\rm(F1)} and {\rm(F2)}, we will prove that
the mild solution of equation \eqref{SecondEquationcauchy} for
$t\in(0,+\infty)$ exists. Precisely, the following theorem holds:
\begin{theorem}\label{wlposednessforF2} For both of the cases {\rm(F1)}
and {\rm(F2)} and $u_0\in C(\mathbb R^2)\cap L^\infty(\mathbb R^2)$,
the problem \eqref{SecondEquationcauchy} has a unique mild solution
$u(x,y,t,u_0,\bar{b})$ for any $t\in(0,+\infty)$ and it is also a
weak solution. Furthermore, it is continuously dependent on the
initial data as in Theorem \ref{wlposedness}.
\end{theorem}

To prove the above theorem, we need the following lemmas on the
boundedness and equicontinuity of solutions of Cauchy problem
\eqref{SecondEquationcauchy} with $\bar{b}$ replaced by a smooth
function $b$:
\begin{lemma}\label{lemma:globalexistenceforf2}
For both of the cases {\rm(F1)} and {\rm(F2)} and $D>0$, there
exists a constant $D'>0$, such that for any $b\in\Lambda(\alpha)$,
$u_0\in C(\mathbb R^2)$ with $\|u_0\|_{L^\infty}\leq D$, the
solution $u(x,y,t,u_0,b)$ of equation \eqref{SecondEquationcauchy},
with $\bar{b}$ replaced by a smooth function $b$, with initial data
$u_0\geq 0$, exists for $t\in (0,\infty)$ satisfies \[
|u(x,y,t,u_0,b)|\leq D', \hbox{ for any }(x,y)\in\mathbb
R^2,\,\,t>0.
\]
\end{lemma}
\begin{lemma}\label{Equicontinuityforf2}
Let $u(x,y,t,u_0,b)$ be as in Lemma
\ref{lemma:globalexistenceforf2}. Then for both of the cases
{\rm(F1)} and {\rm(F2)} and any constants $\epsilon>0, \,M>0,$  it
holds that $\{u(x,y,t,u_0,b)\}_{u_0\in C(\mathbb
R^2),\|u_0\|{L^\infty}\leq M, b\in\Lambda(\alpha)}$ is uniformly
equicontinuous in $(x,y,t)\in \mathbb{R}^2\times[\epsilon,\infty)$.
\end{lemma}

\subsection{Proof of Lemmas \ref{Equicontinuity}, \ref{lemma:globalexistenceforf2} and
\ref{Equicontinuityforf2}}\label{subsection:lemmasforwlposedness}
The equicontinuity result in the Lemma \ref{Equicontinuity} is
standard, but what is important is that the estimate is uniform in
$b\in\Lambda(\alpha)$. For the convenience of the reader, we give an
outline of the proof. We begin with the following two lemmas:
\begin{lemma}\label{lm:estimateforlemmainsection6}
Let $G(x,y,t)$ be the heat kernel defined in \eqref{heatkernel}.
Then there exists a constant $C>0$ such that
\[
\ \sup_{x,y\in\R}\big|\int_{\R}\int_{\mathbb
R}G(x-x',y-y',\tau)\bar{b}(x')\eta(x',y')dx'dy'\big|\leq{C}\alpha\big(1+\frac{L}{\sqrt{\tau}}\big)\|\eta\|_{L^\infty}
\] for any $\tau>0$, $\bar{b}\in\overline{\Lambda}(\alpha)$ and $\eta\in{L}^\infty(\R^2)$.
\end{lemma}
\begin{proof}
Denote by $a(x,y)$ the integral on the left-hand side of the above
inequality. Note that $G(x,y,t)=G(x,t)G(y,t)$, where
\[
\ G(x,t):=\frac{1}{\sqrt{4\pi t}}\exp({-\frac{x^2}{4t}}).
\] Hence $a(x,y)$ can be written as
\[
\ a(x,y)=\int_{\R}G(x-x',\tau)\bar{b}(x')\int_{\mathbb
R}G(y-y',\tau)\eta(x',y')dy'dx'.
\]
By Lemma \ref{lemma:A}, we have
\[
\
|a(x,y)|\leq\alpha{L}\|\eta\|_{L^\infty}\sum_{k=-\infty}^\infty\max_{0\leq{x}\leq{L}}G(x-r-kL,\tau).
\]
Considering that $G(x,\tau)$ is monotone increasing in $x<0$ and
monotone decreasing in $x>0$, we easily find that
\begin{equation}\label{G-sum}
\begin{split}
\
\sum_{k=-\infty}^\infty\max_{0\leq{x}\leq{L}}G(x-r-kL,\tau)&\leq\frac{1}{L}\int_{\R}G(x-r,\tau)dx+2\max_{y\in\R}G(y,\tau)\\
&=\frac{1}{L}+\frac{1}{\sqrt{\pi\tau}}.
\end{split}
\end{equation} Therefore,
\[
\
|a(x,y)|\leq\alpha\|\eta\|_{L^\infty}\big(1+\frac{L}{\sqrt{\pi\tau}}\big).
\] The lemma is proven.
\end{proof}
In what follows, for simplicity, we denote
$b(x)f(u(x,y,t))+g(u(x,y,t))$ by $A(x,y,t)$.
\begin{proof}[Proof of Lemma \ref{Equicontinuity}]
First we prove the existence of $T$. Let $T_b$ be the maximum time
such that the solution $\|u(x,y,t,u_0,b)\|_{L^\infty}\leq 2M$, for
$t\in(0,T_b)$. We need to prove that we can choose a constant
$T'$, which is independent of the choice of $b$, such that
$T_b\geq T'$.

Since $u(x,y,t,u_0,b)$ can be written as
\[
\begin{split}
u(x,y,t,u_0,b)=&\int_{\R}\int_\R G(x-x',y-y',t)u_0(x',y')dx'dy'
\\&+\int_0^t\int_{\R}\int_{\R}G(x-x',y-y',s) A(x',y',t-s)dx'dy'ds.
\end{split}
\]
By Lemma \ref{lm:estimateforlemmainsection6}, we have
\[
\ \sup_{x,y\in\R}|u(x,y,t,u_0,b)|\leq{C}\alpha{K}(t+L\sqrt{t}),
\] where $C$ and $K$ are dependent on $M$ and independent of $b\in\Lambda(\alpha)$. Hence, there exists a
constant $T'>0$ such that $T_b\geq T'$.

We show the uniform equicontinuity of
$$\{u(x,y,t,u_0,b)\}_{u_0\in
C(\mathbb R^2),\|u_0\|_\infty\leq M, b\in\Lambda(\alpha)}$$ in
$(x,y,t)\in\R^2\times[\epsilon,T']$ for any fixed constant $M>0$.
 We still express $u(x,y,t,u_0,b)$ as
\[
\begin{split}
u(x,y,t,u_0,b)=&\int_{\R}\int_\R G(x-x',y-y',t)u_0(x',y')dx'dy'
\\&+\int_0^t\int_{\R}\int_{\R}G(x-x',y-y',s)A(x',y',t-s)dx'dy'ds.
\end{split}
\]
The uniform equicontinuity of the first term on the right-hand
side of the above equation is easily seen. In what follows we
prove the uniform equicontinuity of the second term, which we
denote by $\tilde{u}(x,y,t)$. We choose $t^*\in(0,\epsilon]$ and
split $\tilde{u}(x,y,t)$ as $\tilde{u}(x,y,t)=p(x,y,t)+w(x,y,t)$,
where
\[
\begin{split}
&p(x,y,t):=\int_0^{t^*}\int_{\R}\int_\R G(x-x',y-y',s)A(x',y',t-s)dx'dy'ds,\\
&w(x,y,t):=\int_{t^*}^t\int_{\R}\int_\R G(x-x',y-y',s)A(x',y',t-s)dx'dy'ds.\\
\end{split}
\] By Lemma \ref{lm:estimateforlemmainsection6}, we have
\[
\ \sup_{x,y\in\R}|p(x,y,t)|\leq{C}\alpha{K}(t^*+L\sqrt{t^*}).
\] Therefore,
\[
\ p(x,y,t)\to 0 \,\,\,\,\,\hbox{ as }  t^*\to 0 \,\,\hbox{
uniformly in } \R^2.
\] Next we express $w(x,y,t)$ as
\[
\ w(x,y,t)=\int_0^{t-t^*}\int_{\R}\int_\R
G(x-x',y-y',s)q(x',y',t-s-t^*)dx'dy'ds,
\] where
\[
\ q(x,y,s)=\int_{\R}\int_\R G(x-x',y-y',t^*)A(x',y',s)dx'dy'.
\] By Lemma \ref{lm:estimateforlemmainsection6}, $q(x,y,s)$ is
uniformly bounded as $u_0$ and $b$ vary satisfying $\|u_0\|\leq{M}$,
$b\in\Lambda(\alpha)$. In view of this, we can easily show the
uniform equicontinuity of $w(x,y,t)$ in
$(x,y,t)\in\R^2\times[t^*,T']$ either by direct calculations, or by
observing that $w$ is a solution of the problem
\[
\ \left\{
    \begin{array}{ll}
      w_t=w_{xx}+w_{yy}+q(x,y,t-t^*) & \hbox{ }(x,y\in\R,\,\,t\in[t^*,T']) \\
      w(x,y,t^*)=0, & \hbox{ }
    \end{array}
  \right.
\] and applying the standard parabolic estimates. Therefore $w(x,y,t)$ is uniformly equicontinuous in $\R^2\times[\epsilon,T']$ for any
fixed $t^*\in(0,\epsilon]$. Now we let $t^*\to 0$. Then $p(x,y)\to
0$ uniformly as shown above, hence $w(x,y,t)\to \tilde{u}(x,y,t)$
uniformly in $\R^2\times[\epsilon,T']$. Consequently,
$\tilde{u}(x,y,t)$ is uniform equicontinuous in
$\R^2\times[\epsilon,T']$, therefore the same is true of $u(x,y,t)$.
Letting $T$ be a constant such that $T<T'$, the lemma is proven.
\end{proof}

To prove the uniform boundedness of the mild solutions of
\eqref{SecondEquationcauchy}, we only need to consider the case
{\rm(F2)}. First, we state the equation under the condition
{\rm(F2)}:
\begin{equation}\label{eq:f2cauchy}
\left\{\begin{array}{ll} u_t=u_{xx}+u_{yy}+u(\bar{b}(x)-g(u))
\quad\ \ &((x,y)\in\mathbb R^2,\;t>0),\vspace{5pt}\\
u(x,y,0)=u_0(x,y)\geq 0\quad\ \ &((x,y)\in\mathbb R^2).
\end{array}\right.
\end{equation} We have the following lemma:
\begin{lemma}\label{lm:boundednessofpositivesteadystate}
For case {\rm(F2)}, there exist constants $\rho_1$ and $\rho_2$ with
$0<\rho_1\leq \rho_2$, such that for any ${b}\in{\Lambda}(\alpha)$,
 $\rho \psi_b$ is a subsolution of \eqref{eq:f2cauchy} for
 any $0<\rho<\rho_1$ and $\rho \psi_b$ is a supersolution of \eqref{eq:f2cauchy}
 for any $\rho>\rho_2$ where $\psi_b$ is the principal eigenfunction of the operator
$-L_{0,0,b}$ with $\|\psi_b\|_{L^\infty}=1$.
\end{lemma}
\begin{proof}
We first consider the linearized equation
\[
\ \psi_b''+b(x)\psi_b=-\mu(0,0,b)\psi_b.
\] It is known that there exists a positive periodic function $\psi_b$
satisfying the above equation. By Lemma \ref{LinearPsi}, we know
that $\min_{x\in\R}\psi_b\geq 1/F$ for any $b\in\Lambda(\alpha)$.

Let $\rho>0$ be a given arbitrary number. We consider the following
function:
$$(\rho\psi_b)_{xx}+\rho\psi_b(b(x)-g_1(\rho\psi_b)).$$ Since
$\min\psi_b\geq 1/F$ and
$-\alpha\geq\mu(0,0,b)\geq-\alpha-\alpha^2L^2$, it is not difficult
to see that there exist constants $\rho_1$ and $\rho_2$, which are
independent of the choice of $b$, such that
\begin{equation}\label{subsolution}
\begin{split}
(\rho\psi_b)_{xx}+\rho\psi_b(b(x)-g_1(\rho\psi_b))&=\rho\psi_b(-\mu(0,0,b)-g_1(\rho\psi_b))\\&\geq
\rho\psi_b(-\mu(0,0,b)-g_1(\rho))\\
&>0
\end{split}
\end{equation} for any $\rho<\rho_1$
and
\begin{equation}\label{supsolution}
\begin{split}
(\rho\psi_b)_{xx}+\rho\psi_b(b(x)-g_1(\rho\psi_b))&=\rho\psi_b(-\mu(0,0,b)-g_1(\rho\psi_b))\\&\leq
\rho\psi_b(-\mu(0,0,b)-g_1({\rho}/{F}))\\&<0
\end{split}
\end{equation} for any $\rho>\rho_2$. Hence we complete the proof.
\end{proof}

\begin{proof}[Proof of Lemma \ref{lemma:globalexistenceforf2}]
For case {\rm(F1)}, by the comparison principle, it is not difficult
to see that $u(x,y,t,u_0,b)\leq\max\{\|u_0\|_{L^\infty(\mathbb
R^2)},1\}$.

For case {\rm(F2)}, by the inequality \eqref{supsolution}, if
$\rho>\rho_2$, then $\rho\psi_b$ is a super-solution of equation
\eqref{eq:f2cauchy} with $\bar{b}$ replaced by $b$. Hence, for any
constant $D$, by Lemma \ref{LinearPsi} and the fact that
$\|\psi_b\|_{L^\infty}=1$, one can choose $\rho^*$ be large enough
such that $\rho^*/F>D$ and $\rho^*>\rho_2$, then by the comparison
principle, we have
\[
|u(x,y,t,u_0,b)|\leq \rho^*, \hbox{ for any }(x,y)\in\mathbb
R^2,\,\,t>0.
\] Since $\rho^*$ is independent of the choice of $b\in\Lambda(\alpha)$, by letting $D':=\rho^*$, we complete the proof.

\end{proof}

\begin{proof}[Proof of Lemma \ref{Equicontinuityforf2}]
Since we have shown in Lemma \ref{lemma:globalexistenceforf2} that
$\{u(x,y,t,u_0,b)\}$ is uniformly bounded for $b\in\Lambda(\alpha)$
and $t>0$. Using the same argument in Lemma \ref{Equicontinuity},
the proof of this lemma can be easily obtained.
\end{proof}

\subsection{Proof of well-posedness
theorems}\label{subsection:proofofwlposedness}
\begin{proof}[Proof of Theorem \ref{wlposedness}]  Let $b_n\in\Lambda(\alpha)$ satisfy
$b_n\to \bar{b}$ in the weak$^*$ sense. Then for any given initial
data $u_0(x),$ the problem
\begin{equation}\label{equationforun}
\left\{\begin{array}{ll}
(u_n)_t=(u_n)_{xx}+(u_n)_{yy}+b_n(x)f(u_n)+g(u_n)\\
 u_n(x,y,0)=u_0(x,y)
\end{array}\right.
\end{equation}
has a classical solution $u_n(x,y,t)$ for any $n\in \mathbb N$ which
is also a mild solution. By Lemma \ref{Equicontinuity}, there exists
a constant $T>0$ such that $\{u_n\} \,(n\in\mathbb N)$ are uniformly
bounded for $0<t\leq T$. Let $t_1<t_2<T$ be two positive numbers. By
Lemma \ref{Equicontinuity}, the family of solutions
${u_n}_{n=1}^{\infty}$ is uniformly equicontinuous in
 $(x,y,t)\in\mathbb
R^2\times [t_1,t_2]$. Here the modulus of equicontinuity may depend
on $t_1$ and $t_2$. Applying Arzela-Ascoli theorem, we can get a
subsequence, which we still denote by $\{u_n(x,y,t)\}$ that
converges uniformly in $(x,y,t)\in[-M,M]\times[-M,M]\times[t_1,t_2]$
for every $M>0$ and $0<t_1<t_2.$ The limit function
$u(x,y,t)=\lim_{n\to\infty}u_n(x,y,t)$ is defined for every
$(x,y,t)\in \mathbb R^2\times(0,T)$ and satisfies
\begin{equation}\label{equationwithnonsmoothb}
u_t=u_{xx}+u_{yy}+\bar{b}(x)f(u)+g(u)
\end{equation}
in the distribution sense.

Now, we prove that $u$ is also a mild solution. For simplicity, we
denote $b_n(x)f(u_n(x,y,t))+g(u_n(x,y,t))$ by $A_n(x,y,t)$ and
denote $\bar{b}(x)f(u(x,y,t))+g(u(x,y,t))$ by $\overline{A}(x,y,t)$.
To prove that $u$ is a mild solution, it is sufficient to prove
\[
\begin{split}
\ \int_0^t\int_{\mathbb R}\int_\R &G(x-x',y-y',t-s)\overline{A}(x',y',t)dx'dy'ds\\
=\lim_{n\to\infty}\int_0^t&\int_{\mathbb R}\int_\R
G(x-x',y-y',t-s)A_n(x',y',t)dx'dy'ds.
\end{split}
\]
We consider
\[
\begin{split} \int_0^t\int_{\mathbb R}\int_\R
&G(x-x',y-y',t-s)\big(b_n(x')f(u_n(x',y',s))
\\&+g(u_n(x',y',t))-\bar{b}(x')f(u(x',y',s))-g(u(x',y',s))\big)\,dx'dy'ds,
\end{split}
\]
which we denote by $Z_n(x,y,t).$ First we have

\begin{equation}\label{Toprovemildsol}
\begin{split}
|Z_n(x,y,t)|&\leq\Big|\int_0^t\int_{\mathbb R}\int_\R G(x-x',y-y',t-s)\big(b_n(x')-\bar{b}(x')\big)f(u(x',y',s))\,dx'dy'ds\Big|\\
+\Big|\int_0^t&\int_{\mathbb R}\int_\R
G(x-x',y-y',t-s)b_n(x')\big(f(u(x',y',s))-f(u_n(x',y',s))\big)\,dx'dy'ds\Big|\\
+\Big|\int_0^t&\int_{\mathbb R}\int_\R
G(x-x',y-y',t-s)\big(g(u(x',y',s))-g(u_n(x',y',s))\big)\,dx'dy'ds\Big|.
\end{split}
\end{equation}

It is not difficult to see that the third term of the above sum is
convergent to $0$ as $n\to\infty$. We consider the other two terms.
Since $u(x,y,t)$ is bounded, we have
\[
\begin{split}
&\Big|\int_{\mathbb
R} G(x-x',y-y',t-s)\big(b_n(x')-\bar{b}(x')\big)f(u(x',y',s))\,dx'\Big|\\
&\hspace{30pt}=G(y-y',t-s)\Big|\int_{\mathbb R}
G(x-x',t-s)\big(b_n(x')-\bar{b}(x')\big)f(u(x',y',s))\,dx'\Big|
\\&\hspace{30pt}\leq 2\alpha LG(y-y',t-s)\|f(u)\|_{L^\infty(\mathbb
R^2\times[0,t])}\frac{C_1}{\sqrt{t-s}}
\end{split}
\] for some constant $C_1>0.$  Since
\[
\ \int_0^t\int_{\mathbb
R}\frac{G(y-y',t-s)}{\sqrt{t-s}}dy'ds<+\infty,
\] by Lebesgue convergence theorem,
\[
\ \lim_{n\to\infty}\int_0^t\int_{\mathbb R}\int_\R
G(x-x',y-y',t-s)\big(b_n(x')-\bar{b}(x')\big)(f(u(x',y',s)))\,dx'dy'ds=0.
\]
Similarly,
\[
\ \int_{\mathbb R}\int_\R
G(x-x',y-y',t-s)b_n(z)\big(f(u(x',y',s))-f(u_n(x',y',s))\big)dx'dy'\leq\frac{C_2}{\sqrt{t-s}}
\] for some constant $C_2>0.$ Again, by Lebesgue convergence
theorem,
\[
\ \lim_{n\to\infty}\int_0^t\int_{\mathbb R}\int_\R
G(x-x',y-y',t-s)b_n(x')\big(f(u(x',y',s))-f(u_n(x',y',s))\big)\,dx'dy'ds=0.
\]
Hence
\[
\begin{split}
 u(x,y,t)=&\int_{\mathbb R}\int_\R G(x-x',y-y',t)u_0(x',y')dx'dy'\\+\int_0^t&\int_{\mathbb R}\int_\R G(x-x',y-y',t-s)\Big(\bar{b}(x')f\big(u(x',y',s)\big)+g(u(x',y',s))\Big)dx'dy'ds.
\end{split}
\]
Next we show that $u$ satisfies
\begin{equation}\label{initial condition}
\lim_{t\searrow 0}u(x,y,t)=u_0(x,y)\quad\hbox{ for any } (x,y)
\in\mathbb R^2.
\end{equation}

For any $x\in\mathbb R,$ the integral
\[
\ \int_{\mathbb R}\int_\R G(x-x',y-y',t)u_0(x',y')dx'dy' \to
u_0(x,y),\hbox{ as }t\to 0.
\]
Another integral can be estimated as follows
\[
\begin{split}
 &\Big|\int_0^t\int_{\mathbb R}\int_\R G(x-x',y-y',t-s)\bar{b}(x')f(u(x',y',s))\,dx'dy'ds\Big|\\
 \leq&\int_0^t\frac{D}{\sqrt{t-s}}ds,
\end{split}
\]where $D$ is a constant depending on $\|u(\cdot,\cdot,t)\|_{L^\infty}.$ Since $\|u(\cdot,\cdot,t)\|_{L^\infty}$ is
bounded, we have
\[
\ \int_0^t\int_{\mathbb R}\int_\R
G(x-x',y-y',t-s)\bar{b}(x')f(u(x',y',s))\,dx'dy'ds\to 0,\hbox{ as
}t\to 0.
\] Consequently \eqref{initial condition} holds.

Now we prove the continuous dependence on the initial data and the
uniqueness of the mild solution. Let $u_0,\tilde{u}_0\in C(\mathbb
R^2)\cap L^\infty(\mathbb R^2)$ be arbitrary and let $u,\tilde{u}$
be the corresponding mild solutions of \eqref{SecondEquationcauchy},
the existence of which has been proven above. Then
$\omega:=u-\tilde{u}$ satisfies
\begin{equation}\label{w-eq}
\left\{\begin{array}{ll}
\omega_t=\omega_{xx}+\omega_{yy}+m(x,y,t)\omega
\quad\ \ &((x,y)\in\mathbb R^2,\;t>0),\vspace{5pt}\\
\omega(x,y,0)=\omega_0(x,y)\quad\ \ &(x,y\in\mathbb R),
\end{array}\right.
\end{equation} in the weak sense
where
$$m:=\frac{\bar{b}(x)(f(u)-f(\tilde{u}))+g(u)-g(\tilde{u})}{u-\tilde{u}}$$ is a
measure-valued function of $y$ and $t$.
 We can express $\omega$ as
\begin{equation}\label{integral1}
\begin{split}
&\omega(x,y,t)=\int_{\mathbb R}\int_\R G(x-x',y-y',t)\omega_0(x',y')dx'dy'\\
&\hspace{65pt}+\int_0^t\int_{\mathbb R}\int_\R
G(x-x',y-y',t-s)m(x',y',s)\omega(x',y',s)\,dx'dy'\,ds.
\end{split}
\end{equation}
Define
\[
\rho(t)=\Vert \omega(\cdot,\cdot,t)\Vert_{L^\infty}.
\]
By the local Lipschitz continuity of $f(u)$ and $g(u)$ and the
boundedness of $u$ and $\tilde{u}$, there exists a constant $M>0$
such that for any $\tau\leq T$,
$$\Big |\frac{f(u(x,y,\tau))-f(\tilde{u}(x,y,\tau))}{u(x,y,\tau)-\tilde{u}(x,y,\tau)}\Big |,\,\,\Big |\frac{g(u(x,y,\tau))-g(\tilde{u}(x,y,\tau))}{u(x,y,\tau)-\tilde{u}(x,y,\tau)}\Big |\leq M.$$
Then, by Lemma \ref{lm:estimateforlemmainsection6}, we have
\[
\begin{split}
\Big |\int_{\mathbb R}\int_\R
G(&x-x',y-y',t-s)m(x',y',s)\omega(x',y',s)\,dx'dy'\Big |\\&\leq
M(C\alpha(1+\frac{L}{\sqrt{t-s}})+1)\rho(s).
\end{split}
\]
Hence, there exists a constant $M'>0$
\begin{equation}\label{ineq}
\rho(t)\leq \rho(0)+M'\int_0^t \frac{\rho(s)\,ds}{\sqrt{4\pi
(t-s)}},
\end{equation}
where $M$ is a constant such that $\Vert m(\cdot,\cdot,t)\Vert\leq
M$ for $t\geq 0$. By Lemma 7.7 of Alfaro, Hilhorst, Matano
\cite{s3}, it follows that
\begin{equation}\label{est-rho}
\rho(t)\leq e^{M'^2t/4}\Big(1+\frac{M'}{\sqrt{4\pi}} \int_0^t
\frac{e^{-M'^2s/4}}{\sqrt{s}}\,ds\Big)\rho(0)
=O\Big(e^{M'^2t/4}\big(1+\sqrt{t}\,\big)\rho(0)\Big).
\end{equation}
Consequently
\[
\ \|u(\cdot,\cdot,t)-\tilde{u}(\cdot,\cdot,t)\|_{L^\infty}\leq
e^{M'^2t/4}\Big(1+\frac{M'}{\sqrt{4\pi}} \int_0^t
\frac{e^{-M'^2s/4}}{\sqrt{s}}\,ds\Big)\|u_0-\tilde{u}_0\|_{L^\infty}.
\]
This proves the continuous dependence on the initial data and the
uniqueness of the mild solution.
\end{proof}
\begin{proof}[Proof of Theorem \ref{wlposednessforF2}]
By Lemma \ref{lemma:globalexistenceforf2}, this theorem is an
immediate deduction of Theorem \ref{wlposedness}.
\end{proof}

In the proof of Theorem \ref{wlposedness}, notice the fact that the
mild solution $u(x,y,t)$ is unique. Hence, it is not difficult to
get the following convergence result. If $b_n\in\Lambda(\alpha)$
satisfies $b_n\to\bar{b}$ in the weak$^*$ sense, then $u_n(x,y,t)\to
u(x,y,t)$ uniformly in $x,y\in\R$ and locally uniformly in
$t\in(0,\infty)$ as $n\to\infty$. Consequently, the following
corollary can be obtained. Here, we omit the proof.
\begin{corol}\label{corol:contdeonini}
Let $u(x,y,t,u_0,b)$ be as in Lemma
\ref{lemma:globalexistenceforf2}. If $b_n\in\Lambda(\alpha)$
satisfies $b_n\to\bar{b}$ in the weak$^*$ sense and $u_n^0\in
C(\R^2)\cap L^\infty(\R^2)$ satisfies $u_n^0\to u_0$ in $L^\infty$
norm, then $u(x,y,t,u_n^0,b_n)\to u(x,y,t,u_n^0,\bar{b})$ as $n\to
+\infty$.
\end{corol}

Since comparison principle holds for equation
\eqref{equationforun}, the following proposition is an immediate
result.
\begin{pro}\label{comparisonprinciple}
For any initial data $u_0,v_0\in C(\R^2)\cap L^\infty(\R^2)$ with
$u_0\geq v_0\geq 0$, if the solutions of
\eqref{SecondEquationcauchy} $u(x,y,t,u_0)$ and $u(x,y,t,v_0)$ exist
in $x,y\in\R,\,\,t\in(0,T)$, where $T>0$ is a constant, then we have
$u(x,y,t,u_0)\geq u(x,y,t,v_0)$ for $x,y\in\R$, $t\in(0,T)$.
\end{pro}
\section{Well-posedness in local uniform topology and the theory of
semiflow}\label{Localuniformtopolyandsemiflow}
\subsection{Well-posedness in local uniform
topology}\label{subsection:wlposednessinlocaluniformtopology} We
apply the general results of Weinberger \cite{s8}  to show that
$c^*(\theta,\bar{b})$ has the property asserted at the beginning of
Theorem \ref{cstarbequalcestarb2D}. For this purpose, we need to
consider the solution semiflow of \eqref{SecondEquation2} on the
space
$$BC(\mathbb{R}^2)=\{u\,|\,u\in C(\mathbb R^2)\cap L^\infty(\mathbb
R^2)\}$$ with respect to the local uniform topology, where the
convergence $u_n\to u$ means that $u_n(x,y)\to u(x,y)$ uniformly on
any bounded set. Equip $BC(\mathbb{R}^2)$ with the topology of
locally uniform convergence. It is easy to show that this linear
space $BC(\mathbb{R}^2)$ has the following properties:

For each $l,m\in\mathbb Z$ and $n\in\mathbb N,$ define
$$\|u\|^n_{l,m}=\sum\limits_{i,j\in \mathbb{Z}}\frac{\max\limits_{
(x,y)\in\Omega_{l+i,m+j}^n}|u(x,y)|}{2^{|i|+|j|}},$$ where
$\Omega_{i,j}^n:=[iL,(i+n)L]\times[j,(j+n)]$. Then we have
\begin{pro}\label{equivalent norm} For each $l,m\in\mathbb Z$ and $n\in\mathbb
N,$ $\|\cdot\|^n_{l,m}$ is a norm on $BC(\mathbb{R}^2)$ and it
defines a topology equivalent to the local uniform topology on any
subset $\{u\,|\,u\in BC(\mathbb{R}^2), \|u\|_{L^\infty}\leq M\} $
of $BC(\mathbb{R}^2)$ for any $M>0$.

Furthermore, for any $u\in BC(\mathbb R^2),$
$$2^{-|l|-|m|}\|u\|^n_{i,j}\leq \|u\|^{n}_{l+i,m+j} \leq
2^{|l|+|m|}\|u\|^{n}_{i,j},\,\,\,\forall i,j,l,m\in \mathbb{Z},n\in
\mathbb{N},$$ $$\|u\|^1_{l,m}\leq \|u\|^{n}_{l,m} \leq
\sum\limits_{l\leq i\leq l+n-1,m\leq j\leq m+n-1}\|u\|_{i,j}^1
,\,\,\,\forall l,m\in \mathbb{Z}, n\in \mathbb{N}.$$\end{pro}

In what follows, we always use $BC^+(\mathbb{R}^2)$ to denote
$\{\phi\in BC(\mathbb{R}^2)|\phi\geq 0\}$ and $\mathcal{C}_\beta$ to
denote the set $\{\phi\in BC^+(\mathbb{R}^2)| \phi(x,y)\leq
\beta(x,y), \forall x,y\in \mathbb R\}$ for any $\beta\in
BC^+(\mathbb{R}^2)$. We have the following proposition:
\begin{pro}For any $\beta\in BC^+(\mathbb{R}^2)$,
$\mathcal{C}_\beta$ is a bounded closed subset of $BC(\mathbb{R}^2)$
with respect to the norm $\|\cdot\|_{l,m}^n$ for any $l,m\in
\mathbb{Z}, n\in \mathbb{N}$. Moreover, $\mathcal{C}_\beta$ is
complete.
\end{pro}
In the following proposition, we will show that the continuous
dependence on the initial data of the mild solution of
\eqref{SecondEquationcauchy} also holds with respect to the local
uniform topology:
\begin{pro}\label{lcontinuous}For any $M>0$, the mild solutions of
\eqref{SecondEquationcauchy} for both of the cases {\rm(F1)} and
{\rm(F2)} depend on the initial data continuously in the space
$\mathcal{C}_M$. Precisely, for any $\epsilon>0$ and $T>0$, there is
some $\eta>0$ such that for any two solutions $u,v$ of
\eqref{SecondEquationcauchy} with initial data $u_0,$ $v_0$, if
$\|u_0-v_0\|_{0,0}^1\leq \eta$, then
$\|u(\cdot,t)-v(\cdot,t)\|_{0,0}^1\leq \epsilon$ for $t\in [0,T]$.

\end{pro}
\begin{proof}First, consider the case {\rm(F1)}.
Let $\eta=u-v$ and $\eta_0=u_0-v_0$, then
\begin{equation}\label{integral}
\begin{split}
&\eta(x,y,t)=\int_{\mathbb R^2} G(x-z,y-w,t)\eta_0(z,w)dzdw\\
&\hspace{20pt}+\int_0^t\int_{\mathbb
R^2}G(z,w,t-s)\bar{b}(x-z)m(x-z,y-w,s)\eta(x-z,y-w,s)\,dzdw\,ds,
\end{split}
\end{equation}
where $m(x,y,s)=\int_0^1f'(\tau u(x,y,s)+(1-\tau)v(x,y,s))d\tau$. It
is easy to see that $\{|m(x,y,s)|\}_{x,y\in \mathbb{R},s\in
[0,+\infty)}$ is bounded. Without loss of the generality, suppose
that $|m(x,y,s)|\leq 1$.

 Let
$$I_1(t,x,y)=\int_{\mathbb R^2} G(x-z,y-w,t)\eta_0(z,w)dzdw$$ and
$$I_2(t,x,y)=\int_0^t\int_{\mathbb
R^2}G(z,w,t-s)\bar{b}(x-z)m(x-z,y-w,s)\eta(x-z,y-w,s)\,dz\,dw\,ds.$$
First, we consider $I_1$. It is easy to show that for any
$\epsilon>0$ and $T>0$ there is some $\zeta>0$, such that if
$\|\eta_0\|_{0,0}^1=\|I_1(0,\cdot,\cdot)\|_{0,0}^1<\zeta$, then
$\|I_1(t,\cdot,\cdot)\|_{0,0}^1<\epsilon$ for any $t\in [0,T]$. Then
we consider $I_2$.

We know that $G(x,y,t)=G(x,t)G(y,t)$ with
\[
G(x,t):=\frac{1}{\sqrt{4\pi t}}\exp\big(-\frac{x^2}{4t}\big).
\]
We claim that the following inequality holds:
\[ \begin{split}&\Big|\int_{\mathbb
R^2}G(z,w,t-s)\bar{b}(x-z)m(x-z,y-w,s)\eta(x-z,y-w,s)\,dzdw\,\Big|
\\
\leq &\alpha L\sum\limits_{i,j\in
\mathbb{Z}}\int_{j}^{j+1}G(w,t-s)dw\times\max\limits_{z\in
[iL,(i+1)L]}G(z,t-s)\times\max\limits_{(z,w)\in
\Omega_{i,j}^1}|\eta(x-z,y-w,s)|.\end{split}\] In fact, for any
$b\in \Lambda(\alpha)$, this inequality is obvious. For $\bar{b}\in
\bar{\Lambda}(\alpha)$, let $b_n\in \Lambda(\alpha)$ with $b_n\to
\bar b$, as in the proof of Lemma 2.2 in \cite{llm}, we can prove
this claim.

Hence, we have\[\begin{split}&\max\limits_{(x,y)\in
\Omega_{k,k'}^1}|I_2(t,x,y)|\\
\leq & \alpha L\int_0^t\sum\limits_{i,j\in
\mathbb{Z}}\int_{j}^{j+1}G(w,t-s)dw\times\max\limits_{z\in
[iL,(i+1)L]}G(z,t-s)\times\max\limits_{(z,w)\in
\Omega_{k-i-1,k'-j-1}^2}|\eta(z,w,s)|ds
\end{split}\]
This implies that \[\begin{split}& \|I_2(t,\cdot,\cdot)\|_{0,0}^1 \\
\leq & \alpha L\int_0^t\Big(\sum\limits_{i,j\in
\mathbb{Z}}\int_{j}^{j+1}G(w,t-s)dw\times\max\limits_{z\in
[iL,(i+1)L]}G(z,t-s)\times\\ & \ \ \ \ \ \ \ \sum\limits_{k,k'\in
\mathbb{Z}}2^{-|k|-|k'|}\max\limits_{(z,w)\in
\Omega_{k-i-1,k'-j-1}^2}|\eta(z,w,s)|\Big)ds\\
 =& \alpha L\int_0^t\Big(\sum\limits_{i,j\in \mathbb{Z}}\int_{j}^{j+1}G(w,t-s)dw\times\max\limits_{z\in
[iL,(i+1)L]}G(z,t-s)\times\|\eta(\cdot,s)\|_{-i-1,-j-1}^2\Big)ds\\
\leq & \alpha L\int_0^t\Big(\sum\limits_{i,j\in
\mathbb{Z}}\int_{j}^{j+1}G(w,t-s)dw\times\max\limits_{z\in
[iL,(i+1)L]}G(z,t-s)\times
2^{|i|+|j|+6}\|\eta(\cdot,s)\|_{0,0}^1\Big)ds \\
=& 64\alpha L \|\eta(\cdot,s)\|_{0,0}^1 \Big(\sum\limits_{i\in
\mathbb{Z}}2^{|i|}\max\limits_{z\in
[iL,(i+1)L]}G(z,t-s)\Big)\Big(\sum\limits_{j\in
\mathbb{Z}}2^{|j|}\int_{j}^{j+1}G(w,t-s)dw\Big).\end{split}
\]
Since $2^{|j|}\leq e^x+e^{-x}$ for $x\in[j,j+1]$,
$$\sum\limits_{j\in
\mathbb{Z}}2^{|j|}\int_{j}^{j+1}G(w,t-s)dw<\int_{\mathbb{R}}(e^w+e^{-w})G(w,t-s)dw=2e^{t-s},$$
and
$$\sum\limits_{i\in \mathbb{Z}}2^{|i|}\max\limits_{z\in
[iL,(i+1)L]}G(z,t)\leq 4(e^{(t-s)(\log 2)^2}+\frac{e^{(t-s)(\log
2)^2}}{\sqrt{\pi(t)}})$$ (see the proof of Proposition 5.3 in
\cite{llm}) for $t\in (0,T]$, we have
$$\|\eta(\cdot,t)\|_{0,0}^1\leq \epsilon
+\int_0^t\frac{C}{\sqrt{t-s}}\big\|\eta(\cdot,s)\big\|_{0,0}^1ds$$
for some constant $C$ provided $\|\eta_0\|_{0,0}^1\leq \eta$. It
follows that
\[
\|\eta(\cdot,\cdot,t)\|_{0,0}^1\leq
e^{M^2t/4}\Big(1+\frac{M}{\sqrt{4\pi}} \int_0^t
\frac{e^{-M^2s/4}}{\sqrt{s}}\,ds\Big)\epsilon
=O\Big(e^{M^2t/4}\big(1+\sqrt{t}\,\big)\epsilon\Big)
\]
from Lemma 7.7 of Alfaro, Hilhorst, Matano \cite{s3}. Now consider
the case {\rm(F2)}. From Lemma \ref{lemma:globalexistenceforf2}, we
know that for any $M>0$, there is some $M'>0$ such that
$u(\cdot,t)\in \mathcal{C}_{M'}$ provided $u_0\in \mathcal{C}_M$.
The rest part of proof is same to that for case {\rm(F1)}. We
complete the proof.
\end{proof}
\subsection{Steady-states of the equation}
In this subsection, we prove the existence and uniqueness of the
positive steady-state of \eqref{SecondEquationcauchy}. More
precisely, we consider the equation
\begin{equation}\label{eq:steadystate}
P_{xx}+\bar{b}(x)f(P)+g(P)=0 \hbox{    in }\mathbb{R}
\end{equation}
in the mild sense under the periodicity condition $P(x\pm L)\equiv
P(x)$.
\begin{pro}
Let {\rm(F1)} hold. Then $P\equiv 1$ is the unique $L$-periodic
positive steady-state of equation \eqref{eq:steadystate}.
\end{pro}
\begin{proof}
It is obvious that $P\equiv 1$ is a positive mild solution. The only
thing we need to do is to prove the uniqueness. We suppose that
there is another positive steady-state $P'>0$. If $P'\geq 1$,
$P'\not\equiv 1$, let $u(x,t,u_0)$ be the mild solution of
\eqref{SecondEquationcauchy} with initial data $u_0= P'$. Then by
comparison principle, $u(x,t,u_0)<\max_{x\in\R}{P'}$ for $t>0$. It
is a contradiction with the fact that $P'$ is a steady-state.
Similarly, if $\min_{x\in\R}P'(x)< 1$, let $u(x,t,u_0)$ be the mild
solution of \eqref{SecondEquationcauchy} with initial data $u_0=
P'$. Then by comparison principle, $u(x,t,u_0)>\min_{x\in\R}{P'}$
for $t>0$. It is a contradiction with the fact that $P'$ is a
steady-state. We complete the proof.
\end{proof}
\begin{pro}
Let {\rm(F2)} hold. Then there exists a positive $L$-periodic mild
solution of \eqref{eq:steadystate}. Furthermore, such a solution is
unique.
\end{pro}
\begin{proof}
Consider an approximating sequence $b_n\to \bar{b}$, where
$b_n\in\Lambda(\alpha)$ $(n=1,2,3,...)$ and the convergence is in
the weak$^*$ sense (see Definition \ref{DefofLamdaBarAlpha}). For
each $b_n$, by Lemma \ref{lm:boundednessofpositivesteadystate} and
classical results of parabolic equations, since a pair of
sub-solution and super-solution can be found, a positive periodic
steady-state $P_n$ exists. Furthermore, there exist constants
$\delta_\alpha>0$ and $M_\alpha>0$, such that
$$\delta_\alpha<\|{P}_n\|_{L^\infty}<M_\alpha.$$ Since
$\{\|P_n\|_{L^\infty}\}$ are uniformly bounded, it is easy to see
$\{\|P_n\|_{H^1(0,L)}\}$ are uniformly bounded. By embedding
theorem, $H^1(0,L)\to C[0,L]$ is a compact embedding, so $\{P_n\}$
is a family of uniformly equicontinuous functions. Hence, there
exists a strongly convergent subsequence which we still denote by
$P_n$ such that there exists a function $\bar{P}$ satisfying that
$P_n\to\bar{P}$ in $L^2(0,L)$ and $L^\infty$ topology. Furthermore,
$\bar{P}$ satisfies
\[
\ \bar{P}_{xx}+\bar{P}(\bar{b}(x)-g_1(\bar{P}))=0
\] in the weak sense. Now we show that $\bar{P}$ is a mild solution.
We recall that every function $P_n$ can be written as
\[
\begin{split}
 P_n(x)=&\int_{\R}G(x-x')P_n(x')dx'\\&+\int_0^t\int_{\R}G(x-x',t-s)\Big(P_n(x')\big(b(x')-g_1(P_n(x'))\big)\Big)dx'ds.
\end{split}
\]  Recall that $P_n$ is convergent to $\bar{P}$ uniformly in
$L^\infty$ topology. By Lemma \ref{cor:A}, it is not difficult to
get
\[
\begin{split}
 \bar{P}(x)=&\int_{\R}G(x-x')\bar{P}(x')dx'\\&+\int_0^t\int_{\R}G(x-x',t-s)\Big(\bar{P}(x')\big(b(x')-g_1(\bar{P}(x'))\big)\Big)dx'ds.
\end{split}
\] Next we show that $\bar{P}$ is unique. Let $\tilde{P}>0$ be another
steady-state. Let $$\lambda_1:=\sup\{\lambda\in\R\,|\,\lambda
\tilde{P}<\bar{P}\}.$$  We claim that $\lambda_1\geq 1$.  If
$\lambda_1<1$, let $W:=\lambda_1\tilde{P}$. By a direct computation,
since $g_1'(u)>0$ and $W/\lambda_1>W$, it is easy to get
\[
\ W_{xx}+W(\bar{b}(x)-g_1(W))=W(g_1(\frac{W}{\lambda_1})-g_1(W))>0
\] in the distribution sense. Let $V:=\bar{P}-W$, then $V\geq 0$ and there exists a $x_0\in \R$ such that $V(x_0)=0$.
Furthermore,
\[
 V_{xx}+\bar{b}(x)V\leq -W(g_1(\frac{W}{\lambda_1})-g_1(W))-(\bar{P}g_1(\bar{P})-Wg_1(W)).
\] Since $\lambda_1<1$ and $g'_1>0$, there exist $\delta>0$ and $\eta>0$, such that $V_{xx}+\bar{b}(x)V\leq -\eta<0$ in $(x_0-\delta,x_0+\delta)$.
Furthermore $\bar{b}(x)V\geq 0$ in the distribution sense. We get
$V_{xx}\leq -\eta$ in $(x_0-\delta,x_0+\delta)$ in the distribution
sense. This means that $V(x)$ is strictly concave in the interval
$(x_0-\delta,x_0+\delta)$. Therefore $V(x)$ cannot attain a local
minimum in this interval, contradicting the fact that $V\geq 0$,
$V(x_0)=0$. This contradiction shows that $\lambda_1\geq 1$.
Consequently,
\[
\ \tilde{P}\leq\lambda_1\tilde{P}\leq\bar{P}.
\] Similarly, argument shows $\bar{P}\leq \tilde{P}$, hence $\tilde{P}=\bar{P}$.
\end{proof}
\subsection{Semiflow of the mild solutions and existence of minimal wave
speeds}\label{subsection:existencefominimalspeed} For both of the
cases {\rm(F1)} and {\rm(F2)}, we define an operator
$Q:BC^+(\mathbb{R}^2)\times \mathbb R^+\rightarrow
BC^+(\mathbb{R}^2)$ by
\[
\ Q_t(u_0)(x,y)=Q(u_0,t)(x,y)=u(x,y,t,u_0),
\] and $u=u(x,y,t,u_0)$ is the mild solutions of
\eqref{SecondEquationcauchy} for the case {\rm(F1)} or {\rm(F2)}.
The following proposition shows that $Q$ is a semiflow on the space
$BC^+(\mathbb{R}^2)$ with respect to the local uniform topology, in
other words, with respect to the norm $\|\cdot\|_{0,0}^1$:
\begin{pro}\label{QIsSemiFlow}
$Q$ is a semiflow on $\mathcal {C}_P$ in the following sense:
\begin{enumerate}
\item[{\rm(a)}] $Q_0(u_0)=u_0$,
\item[{\rm(b)}] $Q_{t_1+t_2}=Q_{t_1}\circ Q_{t_2}$,
\item[{\rm(c)}] $Q:(u_0,t)\to Q_t(u_0)$ is continuous in $(u_0,t)$.
\end{enumerate}
Moreover, $Q$ has the following properties:\begin{enumerate}

\item[{\rm(d)}] For each $T>0$ and $M>0$, the family of maps $\{Q_t\}:\mathcal{C}_M\to BC^+(\mathbb{R}^2),\,0\leq t\leq
T$ is equicontinuous.
\item[{\rm(e)}] For any $t,M>0,$
$\{Q_t(u_0)\,|\,u_0\in \mathcal{C}_M\}$ is precompact in
$BC^+(\mathbb{R}^2)$ with respect to the local uniform topology.
\end{enumerate}
\end{pro}
The proof of Proposition \ref{QIsSemiFlow} is essentially the same
as that of Proposition 5.4 in \cite{llm}, so we omit the proof.
\begin{lemma}[Monostability for {\rm (F1)}]\label{monostable} Let {\rm{\rm(F1)}}
hold. Then $u\equiv 1$ and $u\equiv 0$ are both steady-states of
\eqref{SecondEquation2} in the mild sense. Furthermore, for
continuous initial data $u_0$ with $0\leq u_0\leq 1,$ $u_0\not\equiv
0$ and $u_0(x,y)\equiv u_0(x+L,0)$, we have $u(x,y,t,u_0)\to 1$
uniformly in $x,y\in\R$ as $t\to\infty.$
\end{lemma}
The above lemma is a slightly extended version of Lemma 5.5 in
\cite{llm}, in which the same result is proved for the case
$f(u)=u(1-u)$. The idea of the proof is to use the fact that any
constant function lying between $0$ and $1$ is a subsolution.
Since the proof for a general $f(u)$ is the same, we omit the
proof of the above lemma. On the other hand, the proof of
monostability for the case {\rm(F2)} is a little different. In
this case, the constant function is no longer a subsolution,
therefore the same argument does not work. In order to prove the
monostability  for the case {\rm (F2)}, we will use the strong
comparison principle and the
 sublinearity of $Q$, which we state below.
\begin{pro}[Strong comparison principle]\label{SCP}
Let {\rm(F2)} hold. Then for any initial data $u_0,v_0\in
BC^+(\mathbb{R}^2)$ with $u_0\geq v_0$ and $u_0\not\equiv v_0$,
$u(x,y,t,u_0)>u(x,y,t,v_0)$ for any $t>0,x,y\in \R$.
\end{pro}
\begin{proof}
Choose $b_n\in\Lambda(\alpha)$ with $b_n\to \bar b$ in the weak$^*$
sense. Let $u_n(x,y,t,u_0)$ be the solution of
\[
\ (u_n)_t=(u_n)_{xx}+(u_n)_{yy}+u_n(b_n(x)-g_1(u_n))\,\,
\,\,\,(x,y\in\R,\,\,t>0)
\] with initial data $u_0$. Then there is
some $M>0$ such that $$0\leq u_n(x,y,t,u_0), u_n(x,y,t,v_0)\leq M,$$
 for any $t>0$, $x,y\in \R$ by Lemma
\ref{lemma:globalexistenceforf2}. Let
$w(x,y,t)=u(x,y,t,u_0)-u(x,y,t,v_0)$ and
$w_n(x,y,t)=u_n(x,y,t,u_0)-u_n(x,y,t,v_0)$. By the comparison
principle, $w_n(x,y,t)\geq 0$. Consequently, we have $w_n(x,y,t)\geq
v(x,y,t)>0$ for any $x,y\in \R$, $t>0$, where $v(x,y,t)$ is the
solution of
\[
\left\{\begin{array}{ll} v_t=v_{xx}+v_{yy}-(g_1(M)+M_1)v
\quad\ \ &(x,y\in\mathbb R,\;t>0),\vspace{5pt}\\
v(x,y,0)=u_0(x,y)-v_0(x,y)\geq 0\quad\ \ &(x,y\in\mathbb R),
\end{array}\right.
\] with $M_1:=\max_{\xi\in[0,M]}|g'_1(\xi)|$. In view of this and
the fact that $\lim_{n\to\infty}w_n=w$ (see Subsection
\ref{subsection:proofofwlposedness}), we obtain $w(x,y,t)\geq
v(x,t)>0$ for any $x,y\in \R, t>0$. This completes the proof.
\end{proof}
\begin{pro}[Strong sublinearity]\label{Sublinearity} Let {\rm(F2)}
hold. Then for any $u_0\in BC^+(\mathbb{R}^2)$ with $u_0\not \equiv
0$ and any $\epsilon \in (0,1)$, $u(x,y,t,\epsilon u_0)> \epsilon
u(x,y,t,u_0)$ for any $x,y\in \R, t> 0$.
\end{pro}
\begin{proof}
By Proposition \ref{SCP}, we have $u(x,y,t,u_0)>0$ for $t>0$,
$x,y\in\R$. Let $\tilde{u}(x,y,t):=\epsilon u(x,y,t,u_0)$. Then
$\tilde{u}$ satisfies
\[ \left\{\begin{array}{ll}
\tilde{u}_t=\triangle\tilde{u}+\tilde{u}(\bar{b}(x)-g_1(\tilde{u}))+\tilde{u}(g_1(\tilde{u})-g_1(\frac{\tilde{u}}{\epsilon}))\,\,
  &(x,y\in\mathbb R,\;t>0),\vspace{5pt}\\
\tilde{u}(x,y,0)=\epsilon u_0(x,y)\geq 0\quad\ \ &(x,y\in\mathbb
R),
\end{array}\right.
\] where $\triangle=\partial_x^2+\partial_y^2$.
Since $\tilde{u}>0$, $\epsilon\in(0,1)$ and $g_1'>0$, we have
$(\epsilon u)(g_1(\epsilon u)-g_1(\tilde{u}/\epsilon))<0$ for
$t>0$. Hence
\begin{equation}\label{eq:lowersolution}
\left\{\begin{array}{ll}
(\tilde{u})_t<\triangle\tilde{u}+\tilde{u}(\bar{b}(x)-g_1(\tilde{u}))\,\,
  &(x,y\in\mathbb R,\;t>0),\vspace{5pt}\\
\tilde{u}(x,y,0)=\epsilon u_0(x,y)\geq 0\quad\ \ &(x,y\in\mathbb R).
\end{array}\right.
\end{equation}
 Let $b_n$ and $u_n(x,y,t,u_0)$ be as in the proof of
Proposition \ref{SCP} and let $\tilde{u}_n:=\epsilon
u_n(x,y,t,u_0)$. It is easy to see that $\tilde{u}_n$ satisfies
\eqref{eq:lowersolution} with $\bar{b}$ replaced by $b_n$. By the
classical comparison principle, we obtain $u_n(x,y,t,\epsilon
u_0)\geq \tilde{u}_n(x,y,t)$. Consequently, by Corollary
\ref{corol:contdeonini},
\[
u(x,y,t,\epsilon u_0)=\lim_{n\to\infty}u_n(x,y,t,\epsilon u_0)\geq
\lim_{n\to\infty}\tilde{u}_n(x,y,t)=\tilde{u}(x,y,t).
\]
Now by the inequality \eqref{eq:lowersolution}, it is obvious that
for any $\tau>0$,
\[
\ u(x,y,t,\epsilon u_0)\not\equiv\tilde{u}(x,y,t) \,\,\,\,\hbox{in
}(x,y,t)\in\R^2\times[0,\tau].
\] Therefore, there exists a sequence $\tau_k > 0$ with $\tau_k\to 0$  as $k\to\infty$ such that
\[
\ u(x,y,\tau_k,\epsilon u_0)\not\equiv
\tilde{u}(x,y,\tau_k)\,\,\,\hbox{in
}(x,y)\in\R^2,\,\,\,\,\,\,k=1,2,3,\cdots
\]
Recall that $u(x,y,\tau_k,\epsilon u_0)\geq \tilde{u}(x,y,\tau_k)$
for $k=1,2,3,\cdots$. Hence, by Proposition \ref{SCP}, the
inequality
$$u(x,y,t,u(x,y,\tau_k,\epsilon u_0))>u(x,y,t,\tilde{u}(x,y,\tau_k))$$
holds for $t>0$, $k=1,2,3,\cdots$ and $x,y \in \R$. On the other
hand,
\[
\begin{split}
 u(x,y,t,\tilde{u}(x,y,\tau_k))&=u(x,y,t,\epsilon u(x,y,\tau_k,u_0))\\&\geq
\epsilon u(x,y,t,u(x,y,\tau_k,u_0))\\&=\epsilon
u(x,y,t+\tau_k,u_0)\\&=\tilde{u}(x,y,t+\tau_k).
\end{split}
\]  Thus we get $u(x,y,t+\tau_k,\epsilon u_0)>\tilde{u}(x,y,t+\tau_k)$.  By letting $\tau_k\to
0$, we obtain $u(x,y,t,\epsilon u_0)>\tilde{u}(x,y,t)$ for $t>0$,
$x,y\in\R$.
\end{proof}

{}From Propositions \ref{SCP} and \ref{Sublinearity}, we have the
following monostability lemma:
\begin{lemma}[Monostability for {\rm (F2)}]\label{thm:monostability}
Let {\rm(F2)} hold. Let $P$ be the unique $L$-periodic positive
(mild) steady-state of \eqref{SecondEquationcauchy}. Then for any
nonnegative $L$-periodic initial data $u_0\in BC^+(\R^2)$ with
$u_0\not\equiv 0$, $u(x,y,t,u_0)\to P(x)$ as $t\to \infty$ uniformly
for $x,y\in \R$.
\end{lemma}
\begin{proof}
%

Here we need to prove that $0$ is an unstable steady-state. In fact,
still use the limit argument, by Lemma
\ref{lm:boundednessofpositivesteadystate}, we can show that for any
sufficiently small positive number $a$, $u(x,y,t,a\psi_{\bar{b}})>
a\psi_{\bar{b}}(x,y)$ for any $t>0$, where $\psi_{\bar{b}}$ is the
principal eigenfunction of $-L_{0,\theta,\bar{b}}$ with
$\|\phi_{\bar{b}}\|_{L^\infty}=1$. This shows that $0$ is unstable.
Furthermore, by Propositions \ref{SCP} and \ref{Sublinearity}, This
theorem can be obtained from \cite[Theorem 2.3.4]{Zhaobook}.
\end{proof}
To unify the cases of (F1) and (F2), we always use $P$ to denote the
positive $L$-periodic steady-state  Summarizing, for any $t>0$,
$Q_t$ has the following properties:
\begin{enumerate}
\item[{\rm(i)}] $Q_t$ is order-preserving in the sense that if
$u_0,v_0\in \mathcal{C}_P$ and $u_0(x,y)\leq v_0(x,y)$ on
$\mathbb{R}$, $Q_t(u_0)(x,y)\leq Q_t(v_0)(x,y)$ on $\mathbb{R}^2$.

\item[{\rm(ii)}] $Q_t(T_{L,s}(u_0))=T_{L,s}(Q_t(u_0))$ where $T_{L,s}$ is a
shift operator defined by $T_{L,s}(u)(x,y)=u(x-L,y-s)$ for any $s\in
\R$.

\item[{\rm(iii)}] $Q_t(0)=0$ and $Q_t(P)=P$. For any $u_0\in
\mathcal{C}_P$ with $T_{L,s}(u_0)=u_0$ for any $s\in \R$ and $u\not
\equiv 0$, $Q_t(u_0)\rightarrow P$ in the space $\mathcal{C}_P$ with
respect to the local uniform topology.

\item[{\rm(iv)}]  Given $T>0$, the family of maps $Q_t: \mathcal{C}_P\to \mathcal{C}_P,\,0\leq t\leq T,$ is uniformly
equicontinuous with respect to the local uniform topology.

\item[{\rm(v)}] For each $t>0$, $Q_t(\mathcal{C}_P)$ is precompact in
$\mathcal{C}_P$ with respect to the local uniform topology.
\end{enumerate}

Thanks to these properties, we can prove the following theorem:
\begin{theorem} \label{theorem5.6}
For both of the cases {\rm(F1)} and {\rm(F2)}, any
$\bar{b}\in\overline{\Lambda}(\alpha),$ and $\theta\in [0,2\pi)$ the
travelling wave in the direction $\theta$ with speed $c$ of
\eqref{SecondEquation2}  exists for any $c\geq c^*(\theta,\bar{b}),$
$c>0$.
\end{theorem}
\begin{proof}
This theorem can be obtained from the results of
Weinberger\cite{s8}.
 In fact, the above properties (i)-(v) are basically the same as
but slightly stronger than Hypotheses 2.1 in \cite{s8}.
Consequently, from Theorem 2.6 of \cite{s8} we can get the existence
of $c^*(\theta,\bar{b}),$ $\bar{b}\in\overline{\Lambda}(\alpha).$
\end{proof}

\subsection{Linearized equation and
semiflow}\label{subsection:linearizedequationandsemiflow}
To prove $c^*(\theta,\bar{b})=c_e^*(\theta,\bar{b})$, we consider
the linearized equation  \[ \ u_t=u_{xx}+u_{yy}+\bar{b}(x)u.
\] We define a linear space $\mathbb X$ by \[ \mathbb{X}:=\{\phi=\sum\limits_{i=1}^4e^{\xi_ix}\phi_i\,|\,
\xi_i\in \mathbb{R},\,\phi_i\in BC(\mathbb R),i=1,2,3,4\}\] and the
subset
$$\mathbb{X}_M^\xi:=\{\phi\in \mathbb{X}:|\phi(x)|\leq
Me^{\xi(|x|+|y|)}\}$$ for any $M,\xi>0$. Again we equip $\mathbb{X}$
with the local uniform topology. Arguing similarly to the proof of
Theorem \ref{wlposedness} and Proposition \ref{lcontinuous}, we can
obtain:
\begin{lemma}
For any $\phi\in \mathbb{X}$, the mild solution $u(x,y,t)$ of the
equation
\[ \ u_t=u_{xx}+u_{yy}+\bar{b}(x)u
\] with initial data $u(x,y,0)=\phi(x,y)$ exists for all $t>0$ and is
unique. The mild solution is a weak solution and for any $t>0$,
$u(\cdot,t,\phi)\in \mathbb{X}$. Furthermore, the mild solutions
depend on the initial data continuously with respect the local
uniform topology on $\mathbb{X}_M^\xi$ for any $\xi,M>0$.
\end{lemma}
Define \[ \Phi_t(\phi)(x,y)=u(x,y,t,\phi), \,\,\,\forall
\phi\in\mathbb{X},\] where $u(x,y,t,u_0)$ is the mild solution of \[
\ u_t=u_{xx}+u_{yy}+\bar{b}(x)u
\] with initial data $u(x,y,0)=u_0(x,y).$
We can show the following lemma:
\begin{lemma} For any $t,M,\xi\geq 0$,
$\Phi_t$ is continuous on $\mathbb{X}_M^\xi$ with respect to the
local uniform topology. \end{lemma}

Next, for any $M\geq 0$, let $A^0_M=\mathcal{C}_M$ and
$A^{\xi,\theta}_M=\{e^{\xi (x\cos\theta+y\sin\theta)}u\,|\,u\in
A^0_M\}$ for $\xi\in \mathbb{R}$ and $\theta\in [0,2\pi)$. Then we
have

\begin{lemma}For any $\xi,M\geq 0$, $\theta\in [0,2\pi)$
$\Phi_t(A^{\xi,\theta}_M)$ is precompact in $\mathbb{X}$ with
respect to the local uniform topology .
\end{lemma}


In the following lemma, we show that $\Phi_t$ is {\textbf strongly
order-preserving}:
\begin{lemma}
For any nonnegative $u_0\in \mathbb{X}$ with $u_0\not\equiv 0$ and
any $t>0$, we have $\Phi_t(u_0)(x,y)>0$ for any $x,y\in\mathbb R.$
\end{lemma}

\begin{proof}The proof is similar to that of Lemma 5.10 in
\cite{llm}. Let $u(x,y,t,u_0)$ be the mild solution of
\[
\ u_t=u_{xx}+u_{yy}+\bar{b}(x)u,
\]
with initial data $u_0,$ and let $v(x,y,t,u_0)$ be the solution of
\[
\ v_t=v_{xx}+v_{yy}
\] with initial data $u_0(x,y)$. Take a sequence $\{b_n\}\subset\Lambda(\alpha)$
 converging to $\bar{b}$ in the weak$^*$ sense, and let $u_n(x,y,t,u_0)$ be the solution of
\[
\ u_t=u_{xx}+u_{yy}+b_n(x)u,\,\,\,u(x,y,0)=u_0(x,y).
\] Then, arguing as in the proof of Theorem \ref{wlposedness}, we have $u(x,y,t,u_0)=\lim_{n\to\infty}u_n(x,y,t,u_0).$
Since $u_0\geq 0$, $b_n(x)\geq 0$; we have
\[
\ u_n(x,t,u_0)\geq v(x,t,u_0).
\]
{}From the classical theory of the heat equation, $v(x,y,t,u_0)>0$
for $t>0.$ Consequently $\Phi_t(u_0)(x,y)=u(x,y,t,u_0)\geq
v(x,y,t,u_0)>0.$
\end{proof}

Let $\Gamma=\{a(x,y)\in C(\mathbb{R}^2), a(x,y)=a(x+L,0)\}$. We
equip $\Gamma$ with the local uniform topology, which is also
equivalent to the $L^\infty$ topology on $\Gamma$. For any
$\xi\geq 0$ and $\theta\in [0,2\pi)$ define a linear operator
$\mathcal{L}_t^{\xi,\theta}$ on $\Gamma$ by
\[
 \mathcal{L}_t^{\xi,\theta}(a)=e^{\xi (x\cos\theta+y\sin\theta)} \Phi_t(e^{-\xi(x\cos\theta+y\sin\theta)}a).
\]
{}From the definition and the properties of $\Phi_t$, we have

\begin{lemma}
For any $t>0$, $\xi\geq 0$ and $\theta\in [0,2\pi)$,
$\mathcal{L}_t^{\xi,\theta}:\Gamma \rightarrow \Gamma$ is bounded,
compact and strongly positive.

\end{lemma}
\begin{pro}\label{uniqueeigenvalue} If $\psi\in \Gamma$ is a principal eigenfunction of $-L_{\lambda,\theta,\bar{b}}$, then it is a principal eigenfunction of
$\mathcal{L}_t^{\lambda,\theta}$. Consequently if
$\mu(\lambda,\theta,\bar{b})$ is the principal eigenvalue of
$-L_{\lambda,\theta,\bar{b}}$, then
$\exp(-\mu(\lambda,\theta,\bar{b})+\lambda^2)$ is the principal
eigenvalue of $\mathcal{L}_t^{\lambda,\theta}$.
\end{pro}
\begin{proof}The proof is similar to that of Proposition 5.12 in
\cite{llm}. Since $\psi$ is a principal eigenfunction of
$-L_{\lambda,\theta,\bar{b}},$
\[ \
\psi''-2\lambda\cos \theta
\psi'+\bar{b}(x)\psi=-\mu(\lambda,\theta,\bar{b})\psi\,\,\,\,\,\,\,
\hbox{ in the weak sense}.
\]
Let $\phi=e^{\lambda (x\cos\theta+y\sin\theta)}\psi$. Then the above
formula is equivalent to
\[ \
\phi_{xx}+\phi_{yy}+\bar{b}(x)\phi=(-\mu(\lambda,\theta,\bar{b})+\lambda^2)\phi
\,\,\,\,\,\,\,\hbox{ in the weak sense.}
\]
It is easy to show that
$\exp((-\mu(\lambda,\theta,\bar{b})+\lambda^2)t)\phi$ is the mild
solution of $u_t=u_{xx}+u_{yy}+\bar{b}u$ with $u_0(x)=\phi(x,y).$
Hence,
\[\Phi_t(\phi)=\exp((-\mu(\lambda,\theta,\bar{b})+\lambda^2)t)\phi .\]
Finally, we have \[\mathcal{L}_t^{\lambda,\theta}(\psi)=
\exp((-\mu(\lambda,\theta,\bar{b})+\lambda^2)t)\psi. \]
\end{proof}

\begin{pro}
The principal eigenvalue of $\mathcal{L}_t^{0,\theta}$ is lager than
$1$ for any $t>0$ and $\theta\in [0,2\pi)$.
\end{pro}

Summarizing, for any $t>0$, $\Phi_t$ has the following properties:
\begin{enumerate}
\item[{\rm(I)}] $\Phi_t$ is strongly order-preserving in the sense
that for any $u_0\in BC^+(\mathbb R^2)$ with $u_0\not\equiv 0$,
$\Phi_t(u_0)(x,y)>0,$ for any $x,y\in\mathbb R.$

\item[{\rm(II)}] $\Phi_t(T_{L,s}(u_0))=T_{L,s}(\Phi_t(u_0))$ for any $s\in \R$.

\item[{\rm(III)}]  For any $t>0$,$\xi\geq 0$ and $\theta\in [0,2\pi)$ the linear operator
$\mathcal{L}_t^{\xi,\theta}:\Gamma \rightarrow \Gamma$ is bounded,
compact and strongly positive. Moreover, the principal eigenvalue of
$\mathcal{L}_t^{0,\theta}$ is larger than 1.

\end{enumerate}
In the following lemma, we will show that $Q$ can be dominated by a
linear operator from above or below:
 \begin{lemma}\label{QtLeqPhit}
Let $\Phi_t$ be as defined above and $Q_t$ be as in Proposition
\ref{QIsSemiFlow}. Then for any $u_0\in BC^+(\mathbb R^2)$ we have
$Q_t(u_0)(x,y)\leq \Phi_t(u_0)(x,y)$ for any $t>0, x,y\in \mathbb
R.$

Furthermore, for any $\epsilon$ with $0<\epsilon<1$, define an
operator $\Phi^\epsilon_t$ by
\[
 \ \Phi^\epsilon_t
(u_0)=u^{\epsilon}(x,y,t,u_0)
\] where $u^{\epsilon}(x,y,t,u_0)$ is the mild solution of \[
\
u^{\epsilon}_t=u^\epsilon_{xx}+u^\epsilon_{yy}+(1-\epsilon)\bar{b}(x)u^\epsilon
\] with initial data $u_0\in BC(\mathbb{R}^2)$. Then $\ \Phi^\epsilon_t$ also has the properties {\rm(I)}-{\rm(III)}.
Moreover, for any given $t_0>0,$ $\Phi^\epsilon_{t_0}(u_0)\leq
Q_{t_0}(u_0)$ provided $u_0(x)\geq 0$ and that
$\|u_0\|_{L^\infty}$ is small enough.
\end{lemma}

The proof of this lemma is the same to that of Lemmas 5.14 and 5.15
in \cite{llm}.
\section{Proof of the main results}\label{Proofofresults}
\subsection{Proof of Theorems \ref{cstarbequalcestarb2D}, \ref{Convergenceofcstarbn}, \ref{Themforspreadingspeed2D}, and
\ref{MinimalPlanarSpeed}}\label{subsection:proofofmainresults}
\begin{proof}[Proof of Theorem \ref{cstarbequalcestarb2D} and \ref{Themforspreadingspeed2D}] We have already shown in Theorem \ref{theorem5.6} the
existence of $c^*(\theta,\bar{b})$. Now, for any $t>0,$ we can check
that $\Phi_t$ satisfies the hypotheses of Theorem 2.5 in \cite{s8}
and for any $0<\epsilon<1,$ $\Phi^{\epsilon}_t$ satisfies the
hypotheses of Theorem 2.4 in \cite{s8}. Moreover, for any
$\lambda\geq 0,$ the principal eigenvalue of $u_{xx}+u_{yy}-2\lambda
\cos\theta u_x+(1-\epsilon)\bar{b}(x)u$ under the periodicity
conditions converges to the eigenvalue of $u_{xx}+u_{yy}-2\lambda
\cos\theta u_x+\bar{b}(x)u$ under the same periodicity conditions as
$\epsilon\to 0.$ Hence we get
\[
c^*(\theta,\bar{b})=c_e^*(\theta,\bar{b}),\,\,\,\bar{b}\in\overline{\Lambda}(\alpha).\]
Moreover, the spreading speed in the direction $\theta$
$w(\theta;\bar{b})$ exists and it satisfies
\[
\
w(\theta;\bar{b})=\min_{|\theta-\phi|<\frac{\pi}{2}}c^*(\phi;\bar{b})/\cos(\theta-\phi).
\]
\end{proof}
\begin{proof}[Proof of Theorem \ref{Convergenceofcstarbn}]
By Proposition \ref{pro:convergenceofeigenvalue}, Lemma
\ref{lem:boundednessofmu}, and note the fact that
\begin{equation}\label{eq:c*minmu+lamda^2/lamda}
\
c^*(\theta;\bar{b})=c_e^*(\theta;\bar{b})=\min_{\lambda>0}\frac{-\mu(\lambda,\theta,\bar{b})+\lambda^2}{\lambda},
\end{equation} the assertion of this theorem is an immediate conclusion.
\end{proof}
\begin{proof}[Proof of Corollary \ref{cr:directional-symmetry}]
By the Definition \ref{definition for L}\begin{equation}
-L_{\lambda,\theta,\bar{b}}\psi(x)=-\psi''(x)+2\lambda\cos\theta\psi'(x)-\bar{b}(x)f'(0)\psi(x),
\end{equation} it is easy to see that $-L_{\lambda,\theta,\bar{b}}=-L_{\lambda,-\theta,\bar{b}}$
 and  $-L_{\lambda,\theta,\bar{b}}$ and $-L_{\lambda,\theta+\pi,\bar{b}}$ are adjoint with each other.
Hence, the principal eigenvalues have the following relation
$\mu(\lambda,\theta,\bar{b})=\mu(\lambda,-\theta,\bar{b})=\mu(\lambda,\theta+\pi,\bar{b})$
for any $\lambda> 0, \theta\in [0,2\pi)$ and $\bar b\in \bar\Lambda
(\alpha)$. Use the formula \eqref{eq:c*minmu+lamda^2/lamda} again,
the assertion of this corollary is an immediate conclusion.
 \end{proof}
\begin{proof}[Proof of Theorem \ref{MinimalPlanarSpeed}]
Combing Lemma \ref{lem:eigenvalueofhwiththeta} and equation
\eqref{eq:c*minmu+lamda^2/lamda}, it is easy to see that
$$c^*(\theta;h)>c^*(\theta;b),\,\,\hbox{ for any
}b\in\Lambda(\alpha),\,\theta\in[0,2\pi).$$ Considering Theorems
\ref{cstarbequalcestarb2D} and \ref{Convergenceofcstarbn}, we have
\[
\
c^*(\theta;h)=\max_{\bar{b}\in\overline{\Lambda}(\alpha)}c^*(\theta;\bar{b})=\sup_{b\in\Lambda(\alpha)}c^*(\theta;b).
\]
Consequently,
\[
\
w(\theta;h)=\max_{\bar{b}\in\overline{\Lambda}(\alpha)}w(\theta;\bar{b})=\sup_{b\in\Lambda(\alpha)}w(\theta;b),\,\,
\hbox{ for any }\theta\in[0,2\pi).
\]
\end{proof}
 We omit the proof of the above proposition and lemma, since we
can get these results by the Proposition 2.16 and the Theorem 2.18
in \cite{llm}. 
\subsection{Monotonicity in $\theta$}\label{ss:monotonicity-theta}

In this subsection we prove Theorem \ref{Th:monotone-theta}. We
begin with the following proposition which is adapted from Nadin
\cite{Nadin}, where the same variational formula is given for the
case of smooth coefficients in higher dimensions; see also the
references therein.

\begin{pro}\label{pr:Nadin}
The principal eigenvalue $\mu(\lambda,\theta,\bar{b})$ of
\eqref{identityofmulambdabbar} is given by
\begin{equation}\label{Nadin-formula}
 \mu(\lambda,\theta,\bar{b})=\min_{\eta\in H^1_{per}}
 {\mathcal H}(\eta),
\end{equation}
where
\[
\begin{split}
{\mathcal H}(\eta)&= \frac{1}{\int_{I}\eta^2dx}\left(
\int_{I}\big((\eta')^2-\bar{b}(x)\eta^2\big)dx\right.\\
&\hspace{65pt} \left.+\lambda^2\cos^2\theta\Big(
\int_{I}\eta^2dx-\frac{L^2} {\int_{I}\eta^{-2}dx}\Big)\right),
\end{split}
\]
and $I=[0,L)$, $ H^1_{per}=\left\{\eta\in H^1_{loc}(\R)\mid
\eta(x+L)\equiv \eta(x)\right\}. $
\end{pro}

\begin{proof}
We first show that ${\mathcal H}$ attains a minimum in $H^1_{per}$.
By Schwarz inequality
\begin{equation}\label{Schwarz}
 \int_{I}\eta^2dx-\frac{L^2}{\int_{I}\eta^{-2}dx}\geq 0.
 \end{equation}
Hence
\[
{\mathcal H}(\eta)\geq \frac{1}{\int_{I}\eta^2dx}
\int_{I}\big((\eta')^2-\bar{b}(x)\eta^2\big)dx\geq
\mu(\lambda,0,\bar{b}).
\]
Therefore $\mathcal H$ is bounded from below in $H^1_{per}$. Let
$\eta_k\;(k=1,2,\cdots)$ be a sequence in $H^1_{per}$ such that
$\int_I\eta_k^2dx=1$ and that ${\mathcal H}(\eta_k)\to\inf {\mathcal
H}$. Then $\{\eta_k\}$ is bounded in $H^1_{per}$. Thus we may assume
without loss of generality that $\eta_k$ converges weakly to an
element $\tilde\eta\in H^1_{per}$. Hence $\eta_k$ converges
uniformly to $\tilde\eta$ on $\R$.  In view of this and from Fatou's
lemma, we see that
\[
{\mathcal H}(\tilde\eta)\leq \lim_{k\to\infty}{\mathcal H}(\eta_k)
=\inf_{\eta\in H^1_{per}}{\mathcal H}(\eta).
\]
This proves that ${\mathcal H}$ attains a minimum in $H^1_{per}$.

Next we show that the minimizer $\tilde\eta$ never vanishes. Suppose
the contrary, and assume $\tilde\eta(x_0)=0$ for some $x_0\in\R$.
Then
\[
|\tilde\eta(x)-\tilde\eta(x_0)|\leq \Vert\tilde\eta\Vert_{H^1}
\sqrt{|x-x_0|}\quad\ \hbox{for}\ \ |x-x_0|\leq L.
\]
Hence $\int_{I}\tilde\eta^{-2}dx=\infty$, and
\[
{\mathcal H}(\tilde\eta)= \frac{1}{\int_{I}\tilde\eta^2dx}
\int_{I}\big((\tilde\eta')^2-\bar{b}(x)\tilde\eta^2\big)dx
+\lambda^2\cos^2\theta.
\]
Consequently $\tilde\eta$ is a minimizer of the functional
\[
{\mathcal H}_0(\eta):= \frac{1}{\int_{I}\eta^2dx}
\int_{I}\big((\eta')^2-\bar{b}(x)\eta^2\big)dx.
\]
Hence $\tilde\eta$ is a principal eigenfunction of the operator
$-L_{0,0,\bar{b}}:=-d^2/dx^2-\bar{b}(x)$.

It follows that $\tilde\eta$ is everywhere positive or everywhere
negative, contradicting the assumption that $\tilde\eta(x_0)=0$.
This contradiction proves that $\tilde\eta$ does not vanish; hence
it has a constant sign.

Without loss of generality we may assume that $\tilde\eta>0$. Then,
by periodicity, $\tilde\eta\geq \delta$ for some positive constant
$\delta$.  In view of this, we see that $\tilde\eta$ satisfies the
following Euler-Lagrange equation in the weak sense:
\begin{equation}\label{EL-eta}
-\tilde\eta''-\bar{b}(x)\tilde\eta+\lambda^2\cos^2\theta\Big(1-
\frac{A^2}{\tilde\eta^4}\Big)\tilde\eta={\mathcal H}(\tilde\eta)
\:\!\tilde\eta,
\end{equation}
where $A=L(\int_I \tilde\eta^{-2}dx)^{-1}$. Now we define a function
$\psi=\tilde\eta e^{\lambda(\cos\theta)\xi}$, where $\xi(x)$ is a
periodic function satisfying $\xi'=1-A\tilde\eta^{-2}$. Such a
periodic function exists since $\int_I(1-A\tilde\eta^{-2})dx=0$.
Then $\psi>0$ and it satisfies
\[
-\psi''(x)+2\lambda\cos\theta\psi'(x)-\bar{b}(x)\psi(x)= {\mathcal
H}(\tilde\eta)\psi
\]
in the weak sense.  Thus $\psi$ is the principal eigenfunction of
the operator $-L_{\lambda,\theta,\bar{b}}$.  By the uniqueness of
the principal eigenvalue (see \cite[Proposition 2.13]{llm}),
${\mathcal H}(\tilde\eta)=\mu(\lambda,\theta,\bar{b})$. This
completes the proof of the proposition.
\end{proof}

\begin{remark}
If $\tilde\eta$ is a minimizer of $\mathcal H$ in $H^1_{per}$, then
$\psi^*:=\tilde\eta e^{-\lambda(\cos\theta)\xi}$ is the principal
eigenfunction of the adjoint operator
$-L_{-\lambda,\theta,\bar{b}}$. Consequently
$\tilde\eta=\sqrt{\psi\psi^*}$, as mentioned in \cite{Nadin} for the
smooth case.
\end{remark}

\begin{proof}[Proof of Theorem \ref{Th:monotone-theta}]
If $\bar{b}$ is not a constant, then the minimizer $\tilde\eta$ of
$\mathcal H$ is not a constant fuction, as one can see from the
Euler-Lagrange equation \eqref{EL-eta}. Hence the inequality
\eqref{Schwarz} holds strictly. This and \eqref{Nadin-formula} imply
that $\mu(\lambda,\theta,\bar{b})$ is a strictly increasing function
of $|\cos\theta|$ for every fixed $\lambda>0$. The conclusion of the
theorem follows from this, Theorem \ref{cstarbequalcestarb} and
\eqref{eq:sprspeed}.
\end{proof}
\subsection{Asymptotic speeds for large and small L}\label{subsection:largeLandsmallL}
\begin{proof}[Proof of Theorem \ref{Thm:Lto0}]
By Theorem \ref{Convergenceofcstarbn}, we know that for any
$b\in\Lambda(\alpha)$, it holds that $2\sqrt{\alpha}\leq
c^*(\theta,b)\leq 2\sqrt{\alpha+\alpha^2L^2}$. Noticing that
$c^*(\theta,h)$ is a limit of some sequence $\{c^*(\theta,b_n)\}$,
where $b_n\in\Lambda(\alpha)$. By using the formula
$w(\theta;\bar{b})=\min_{|\theta-\phi|<\frac{\pi}{2}}c^*(\phi;\bar{b})/\cos(\theta-\phi)$,
we have
\[
\ \lim_{L\to 0}c^*(\theta;h)=2\sqrt{\alpha},\,\,\,\,\,\,\lim_{L\to
0}w(\theta;h)=2\sqrt{\alpha}.
\]
\end{proof}
\begin{proof}[Proof of  Theorem \ref{thm:speedforlargel}]
Consider the equation
\[
\
-\psi''+2\lambda\cos\theta\psi'-h(x)\psi=\mu(\lambda,\theta,h)\psi.
\]
By a direct computation, the eigenvalue $\mu(\lambda,\theta,h)$
satisfies
\begin{equation}\label{eq:muhsequation}
\begin{split}
\ &2\sqrt{{\lambda^2\cos^2\theta}-\mu(\lambda,\theta,h)}\\&=\alpha
L(\frac{1}{1-e^{(({\lambda\cos\theta}-\sqrt{{\lambda^2\cos^2\theta}-\mu(\lambda,\theta,h)})L)}}+\frac{1}{e^{({\lambda\cos\theta}+\sqrt{{\lambda^2\cos^2\theta}-\mu(\lambda,\theta,h)})L}}-1).
\end{split}
\end{equation}
For simplicity, we denote $\mu(\lambda,\theta,h)$ by $\mu$.
 that there exist $0<\varepsilon\ll 1$,
 $M>0$, $k>0$, such that for $kL>\lambda>0$,
\[
\ \mu\leq -ML^\varepsilon.
\]  By the proposition \ref{pro:muandl}, it holds that at least for
$\lambda=0$, $\mu\leq -ML^\varepsilon$. If our claim is not true, by
the continuity of $\mu$, it is not difficult to get a contradiction.
Hence
\[
\ ({\lambda\cos\theta}+\sqrt{{\lambda^2\cos^2\theta}-\mu})L\geq
\sqrt{ML^\varepsilon}L\to\infty \,\,\,(\,L\to\infty\,)
\]uniformly in ${\lambda}\in[0,kL]$, $\theta\in(0,2\pi]$,
$\mu\in(-\infty,-ML^\varepsilon)$. On the other hand,
\[
\ ({\lambda\cos\theta}-\sqrt{{\lambda^2\cos^2\theta}-\mu})L\leq
\frac{-ML^{1+\varepsilon}}{{\lambda\cos\theta}+\sqrt{{\lambda^2\cos^2\theta}+ML^\varepsilon}}\to
-\infty\,\,\,(\,L\to\infty\,)
\] uniformly in ${\lambda\cos\theta}\in[0,kL]$.
Hence, by equation \eqref{eq:muhsequation},
\[
\begin{split}
\ 2\sqrt{{\lambda^2\cos^2\theta}-\mu}=&\alpha
L(1+e^{({\lambda\cos\theta}-\sqrt{{\lambda^2\cos^2\theta}-\mu})L}\\&+\frac{1}{e^{({\lambda\cos\theta}+\sqrt{{\lambda^2\cos^2\theta}-\mu})L}-1}+o(e^{({\lambda\cos\theta}-\sqrt{{\lambda^2\cos^2\theta}-\mu})L}))\\
=&\alpha L+o(L^{-1}).
\end{split}
\] Hence
\[
\ \mu=\frac{-\alpha^2L^2}{4}+{\lambda^2\cos^2\theta}+o(L^{-1}).
\]
Consider the inequality
\[
\ -\lambda c+\lambda^2\leq \mu.
\] Let $\tilde{\lambda}:=\frac{\lambda}{L}$,
$\tilde{\mu}:=\frac{\mu}{L^2}$, $\tilde{c}:=\frac{c}{L}$. We
consider the properties of the $\tilde{c}$ when $L$ tends to
infinity. First, by the fact that $\mu<0$, it holds that
\[
\ \tilde{\mu}\to\left\{%
\begin{array}{ll}
    -\frac{1}{4}\alpha^2+\tilde{\lambda}^2\cos^2\theta, & \hbox{ ( $0\leq\tilde{\lambda} < \frac{\alpha}{2\cos\theta}$ ) ;} \\
    0, & \hbox{ ( $\tilde{\lambda}\geq \frac{\alpha}{2\cos\theta}$ ) .} \\
\end{array}%
\right.
\]
By the arguments above, we have the inequality
\[
\
\tilde{\lambda}^2-\tilde{c}\tilde{\lambda}-\tilde{\lambda}^2\cos^2\theta+\frac{1}{4}\alpha^2\leq
0.
\]
Recall the definition of $c^*(\theta;h)$, it is not difficult to
obtain
\[
\ \lim_{L\to\infty}\frac{c^*(\theta;h)}{L}=\left\{%
\begin{array}{ll}
    \frac{\alpha}{2\cos\theta}\,\,\,\, & \hbox{ if }\cos^2\theta\geq\frac{1}{2}; \\
    \alpha\sin\theta\,\,\,\, & \hbox{ if }\cos^2\theta\leq\frac{1}{2}. \\
\end{array}%
\right.
\] Therefore, by the equation \eqref{eq:sprspeed},
\[
\
\lim_{L\to\infty}\frac{w(\theta;h)}{L}=\frac{\alpha}{1+|\cos\theta|}.
\]
\end{proof}
\begin{pro}\label{pro:muandl}
There exist constants $M,\varepsilon>0$, such that
\[
\ \mu(0,\theta,h)\leq -M L^\varepsilon.
\]
\end{pro}
\begin{proof}
For simplicity, we denote $\mu(0,\theta,h)$ by $\mu$. Consider the
equation
\[
\-\psi''-h(x)\psi=\mu\psi.
\]
By a direct computation, one can get the equation
\[
\ 2\sqrt{-\mu}=\alpha
L(\frac{1}{1-e^{-\sqrt{-\mu}L}}+\frac{1}{e^{\sqrt{-\mu}L}-1}).
\] Hence
\[
\
\frac{2\sqrt{-\mu}}{L}=\alpha(\frac{e^{\sqrt{-\mu}L}+1}{e^{\sqrt{-\mu}L}-1})>\alpha.
\] Therefore
\[
\ \mu<-\frac{\alpha^2L^2}{4}.
\] By letting $M:=\alpha^2$, $\varepsilon<2$, we complete the
proof.
\end{proof}
\section{General 2-dimensional case}\label{General2Dimensionalcase}
In this section, we will consider a more general 2-dimensional
equation \begin{equation}\label{2Equation} u_t=
u_{xx}+u_{yy}+{b}(x,y)f(u)+g(u) \quad\ \ x,y\in\mathbb R.
\end{equation}
where $f,g$ satisfy {\rm(F1)} or {\rm(F2)}.

Here, we assume that $b\in \Lambda_{x,y}(\alpha)$ with \[
\begin{split}
\ \Lambda_{x,y}(\alpha):=\{b(x,y) \in C^1(\mathbb R^2)\,|\,
b(x,y)\geq 0,\,b(x,y)=b(x+L_1,y)=b(x,y+L_2) \,\\~and \int_{[0,L_1)}
\int_{[0,L_2)}b(x,y)dxdy=\alpha L_1L_2 \}.
\end{split}
\]
\begin{definition}For any $c>0$, a entire solution of
\eqref{2Equation} is called a travelling wave in the direction
$\theta$ with average speed $c$ if it can be written as
\[
\ u(x,y,t)=\varphi(x,y,x\cos\theta+y\sin\theta-ct)
\] where $\varphi$ satisfies $\varphi(x+L_1,y+L_2,s)\equiv\varphi(x,s)$.
\end{definition}
Similar to the proof of Lemma \ref{thm:monostability}, for both
cases of {\rm(F1)} and {\rm(F2)} it is known that there is a
positive steady-state
 $P(x,y)$ with $P(x,y)=P(x+L_1,y)=P(x,y+L_2)$ such that for any nonnegative function $u_0$ with $u_0(x,y)=u_0(x+L_1,y)=u_0(x,y+L_2)$
  and $u_0\not\equiv 0$, the solution $u(x,y,t,u_0)$ of \eqref{2Equation} with initial data $u_0$ satisfies $\lim_{t\to \infty}u(x,y,t,u_0)=P$
uniformly for $x,y\in \mathbb{R}$, where $P=1$ for
\eqref{2Equation}. Moreover, we have for any $\theta\in [0,2\pi)$
and $b\in \Lambda_{x,y}(\alpha)$ there is some positive number
$c^*(\theta,b)$ such that the travelling wave $u(x,y,t)$ in the
direction $\theta$ with average speed $c$ and $u(x,y,+\infty)=P, \
u(x,y,-\infty)=0$ exists if and only if $c\geq c^*(\theta,b)$. In
this sense, we call $c^*(\theta,b)$ the {\bf minimal speed} of
travelling wave in the direction $\theta$. Moreover, we have the
following estimate of $c^*(\theta,b)$:\[ \
c^*(\theta;b)=\inf\{c>0\,|\,\exists\,\lambda>0,\,\hbox{such that
}\mu(\lambda,\theta,{b})=\lambda^2-\lambda c\}
\] where $\mu(\lambda,\theta,{b})$ is the principal eigenvalue of
the eigenvalue problem
\[-u_{xx}-u_{yy}+2\lambda \cos \theta u_x+2\lambda \sin \theta
u_y-b(x,y)u=\mu u\] with $u(x+L_1,y)=u(x,y+L_2)=u(x,y)$.

Furthermore, let
$w(\theta;{b})=\min_{|\theta-\phi|<\frac{\pi}{2}}c^*(\phi;{b})/\cos(\theta-\phi)$.
Then any solution $u(x,y,t)$ of \eqref{2Equation} with nonnegative
initial data $u_0(x,y)$ with compact support and $u_0(x,y)\not\equiv
0$ satisfies
\[
\ \lim_{t\to\infty}u(x,y,t)=0,\hbox{   on }\{x\cos\theta+y\sin\theta
>ct\}\,\,\,\hbox{  if }c>w(\theta;b),
\]
\[
\ \lim_{t\to\infty}u(x,y,t)=1,\hbox{   on
 }\{x\cos\theta+y\sin\theta <ct\}\,\,\,\hbox{  if }c<w(\theta;b).
\]
In this sense, we call $w(\theta;{b})$ the {\bf spreading speed} in
the direction $\theta$ of \eqref{2Equation}.

As we showed in Section 4, for one-dimensional case,
$\{c^*(b)\}_{b\in\Lambda(\alpha)}$ is bounded. But for 2-dimensional
case, the conclusion is different. Precisely, we have the following
theorem for general periodic function $b(x,y)$:
\begin{theorem}\label{xy}There is a sequence $\{b_n\}\subset
\Lambda_{x,y}(\alpha)$ such that for any $\theta\in [0,2\pi)$,
$\lim_{n\to\infty} w(\theta;{b_n})=\lim_{n\to\infty}
c^*(\theta;{b_n})=+\infty$.
\end{theorem}
To prove Theorem \ref{xy}, the following two propositions are
needed:
\begin{pro}
For any $b\in\Lambda_{x,y}(\alpha)$, $\theta\in [0,2\pi)$ and
$\lambda\in \mathbb{R}$, $\mu(\lambda,\theta,{b})\leq -\alpha$.
\end{pro}
\begin{pro}\label{unbounded}
$\lim_{n\to\infty} \mu(\lambda,\theta,{b_n})=-\infty$ uniformly for
any $\theta\in [0,2\pi)$ and $\lambda$ in any bounded interval.
\end{pro}
\begin{proof}Let $\tilde{\mu}(\lambda,\theta,{b})$ be the principal eigenvalue of
the eigenvalue problem \[-\phi_{xx}-\phi_{yy}+2\lambda \cos \theta
\phi_x+2\lambda \sin \theta \phi_y-b(x,y)\phi=\mu\phi\] with
$\phi\in C_0([0,L_1]\times[0,L_2])$. Since $b$ is a smooth function,
by classical results of the eigenvalue problems with Dirichlet
boundary and periodic boundary, it is well known that
$\mu(\lambda,\theta,{b})<\tilde{\mu}(\lambda,\theta,{b})$ for any
$b\in\Lambda_{x,y}(\alpha)$, $\theta\in [0,2\pi)$ and $\lambda\in
\mathbb{R}$.  Moreover, let $\bar{\mu}({b})$ be the principal
eigenvalue of
\[-\psi_{xx}-\psi_{yy}-b(x,y)\phi=\mu\psi\] with $\psi\in
C_0([0,L_1]\times[0,L_2])$. Then it is easy to check that
$\bar{\mu}({b})=\tilde{\mu}(\lambda,\theta,{b})-\lambda^2$.

Thus, to prove this proposition, we only need to prove that there is
a sequence $\{b_n\}\subset \Lambda_{x,y}(\alpha)$ such that
$\lim_{n\to \infty}\bar{\mu}({b_n})=-\infty$. In fact, there is s
sequence $\{v_n\}\in C_0^1([0,L_1]\times[0,L_2])$ such that
$\|v_n\|_{H^1}=1$ and $\|v_n\|_{L^\infty}\to +\infty$. Therefore,
\[\frac{\int_{[0,L_1]}\int_{[0,L_2]}|\nabla v_n|^2dy dx-\alpha L_1L_2\|(v_n)^2\|_{L^\infty}}{\int_{[0,L_1]}\int_{[0,L_2]}(v_n)^2dydx} \to
-\infty.\]Moreover, we can find $\{b_n\}\subset
\Lambda_{x,y}(\alpha)$ such that \[\alpha
L_1L_2\|(v_n)^2\|_{L^\infty}-\int_{[0,L_1]}\int_{[0,L_2]}b(x,y)(v_n)^2dydx
\to 0.\] This implies that
\[\frac{\int_{[0,L_1]}\int_{[0,L_2]}|\nabla v_n|^2dy dx-\int_{[0,L_1]}\int_{[0,L_2]}b(x,y)(v_n)^2dydx}{\int_{[0,L_1]}\int_{[0,L_2]}(v_n)^2dydx}
\to -\infty.\] Hence, by the variational formula of $\bar{\mu}(b)$:
\[\bar{\mu}(b)=\min _{\psi\in C([0,L_1]\times[0,L_2]),\psi\not=0}\frac{\int_{[0,L_1]}\int_{[0,L_2]}|\nabla \psi|^2dy dx-\int_{[0,L_1]}\int_{[0,L_2]}b(x,y)(\psi)^2dydx}{\int_{[0,L_1]}\int_{[0,L_2]}\psi^2dydx},\]
 we have $\bar{\mu}(b_n)\to -\infty$.\end{proof}
\begin{proof}[Proof of Theorem \ref{xy}]
First, $c^*(\theta,b)$ can be represented as $$
c^*(\theta,b)=\min\limits_{\lambda>0}\frac{-\mu(\lambda,\theta,{b})+\lambda^2}{\lambda}.$$
To be contrary, suppose that there are a positive number $A$ and
some subsequence, still denoted by $\{b_n\}$ and $\theta_n$ such
that $c^*(\theta_n,b_n)<A$. Then there is some $\lambda^*_n<A$, such
that $
c^*(\theta_n,b_n)=\frac{-\mu(\lambda^*_n,\theta_n,{b_n})}{\lambda^*_n}+\lambda^*_n.$
But by Proposition \ref{unbounded},$
-\mu(\lambda^*_n,\theta_n,{b_n})\to -\infty$, a contradiction. Hence
$\lim_{n\to \infty}c^*(\theta,b_n)=\infty$ uniformly for $\theta\in
[0,2\pi)$, so does $w(\theta,b_n)$. This completes our proof.
\end{proof}

\bibliographystyle{amsplain}


\end{document}